\documentclass{article}
\usepackage{graphicx} 
\usepackage[utf8]{inputenc}
\usepackage{float}
\usepackage{url}

\usepackage{amssymb,amsthm,amsmath}

\usepackage[margin=1in]{geometry}
\usepackage{thm-restate}
\usepackage{enumerate}
\usepackage{subcaption}
\usepackage{cite}
\usepackage[ruled,vlined,linesnumbered]{algorithm2e}
 \usepackage{xcolor}
\usepackage{hyperref}
\usepackage{cleveref}
\usepackage{listings}
\hypersetup{
    colorlinks=true,
    linkcolor=blue,
    filecolor=magenta,      
    urlcolor=cyan,
    pdftitle={Overleaf Example},
    pdfpagemode=FullScreen,
    }
\usepackage{bbm}

\usepackage{tikz}
\usetikzlibrary{fit,arrows.meta,positioning,shapes.geometric}

\newtheorem{theorem}{Theorem}

\newtheorem{corollary}{Corollary} 
\usepackage{amsmath}
\usepackage{orcidlink}
\usepackage{framed}

\usepackage[normalem]{ulem} 

\usepackage{marginnote}
\normalmarginpar

\newcommand{\lineofalg}[2]{Line \ref{#1} of \Cref{#2}}

\newcommand{\xb}{\mathbf{x}}
\newcommand{\cb}{\mathbf{c}}
\newcommand{\yb}{\mathbf{y}}
\newcommand{\eb}{\mathbf{e}}
\newcommand{\gb}{\mathbf{g}}

\newcommand{\Hb}{\mathbf{H}}

\newcommand{\sv}{\mathbf{s}}

\newcommand{\ub}{\mathbf{u}}

\newcommand{\Yb}{\mathbf{Y}}
\newcommand{\zb}{\mathbf{z}}

\newcommand{\ellb}{\mathbf{\ell}}

\newcommand{\DD}{\Delta\!\!\!\!\Delta}

\newcommand{\gammainc}{\gamma_{inc}}
\newcommand{\Deltamax}{\Delta_{\max}}

\newcommand{\Reals}{\mathbb{R}}

\newcommand{\Prob}{\mathbb{P}}
\newcommand{\Exp}{\mathbb{E}}
\newcommand{\Ind}{\mathbbm{1}}

\title{Importance Sampling in Expensive Finite-Sum Optimization via Contextual Bandit Methods}
\author{Matt Menickelly\footnote{Argonne National Laboratory, 9700 S. Cass Ave., Lemont, IL 60439, USA.}\orcidlink{0000-0002-2023-0837}}

\begin{document}
\maketitle

\begin{abstract}
In computational science workflows, it is often the case that 1) objective functions for optimization involve multiple simulation outputs, and 2) those simulations can be performed (at least partially) in parallel. 
In this work, we reexamine past work on a class of randomized algorithms, stochastic average model (SAM) methods. SAM methods are conceptually similar to stochastic average gradient methods, and effectively require that only randomized subsets of simulation outputs be locally modeled in each iteration of a model-based optimization method. 
This work focuses on the question of how best to perform this randomization of subset selection, especially in settings where there exists useful side information such as alternative lower-fidelity simulations, pre-trained emulators or domain expertise from humans or AI models.
In particular, we consider the problem of generating sampling distributions for SAM methods as a contextual bandit problem and we apply the Exponential weights algorithm for Exploration and Exploitation with Experts (Exp4). 
We provide some preliminary numerical results on synthetic problems. 
\end{abstract}

\section{Problem Setting}\label{sec:problem} 
Let $F:\Reals^{dim}\to\Reals^p$ denote a vector-valued function of $\xb\in\Reals^{dim}$, with the output vector denoted by $F(\xb) = [F_1(\xb),\dots,F_p(\xb)]^\top$. 
We consider the finite-sum minimization problem
\begin{equation}\label{eq:sum}
\displaystyle\min_{\xb\in\Reals^{dim}} \displaystyle\sum_{i=1}^p F_i(\xb).
\end{equation}

Our primary motivation comes from simulation-based optimization, where the evaluation of $F(\xb)$ is assumed to be computationally expensive. 
As a  motivating example, 
we refer to the problem of density functional theory (DFT) calibration. 
In such problems, each $F_i$ in \cref{eq:sum} is a likelihood term that fits, to experimental values, an observable computed from a DFT simulation (parameterized by $\xb$) of a particular nucleon.
Executing the DFT forward model involves a computationally intensive solver such as \texttt{HFBTHO} \cite{marevic2022axiallydeformed}.
As additional corrections (e.g., particle number projection (PNP) or angular momentum projection (AMP)) are added to the model, the computational complexity of the forward model increases. 
The application of a derivative-free trust-region method, Practical Optimization Using No Derivatives and Exploiting Recognized Structure (\texttt{POUNDERS}, available in the \texttt{IBCDFO} repository \cite{ibcdfo}), to such problems when the likelihood function is expressed as a least-squares minimization problem, has been studied intensively \cite{SWCHAP14}. 
Bollapragada  et al. \cite{Bollapragada2020}
compared the performance of various derivative-free and approximate stochastic gradient-type methods on a class of problems of DFT calibration and demonstrated an empirical preference for model-based derivative-free optimization methods, such as \texttt{POUNDERS}. 

In the motivating DFT example, and in many instances in computational science, it is the case that batches of $\{F_j(\xb): j=1,\dots,p\}$ can be computed independently, and hence, can be computed in parallel. 
As a simplifying assumption in our problem setting throughout this paper, we will suppose that any arbitrary subset of $\{F_j(\xb): j=1,\dots,p\}$ can be computed in parallel without loss of efficiency. 
This assumption is not always realized in practice; in fact, in the DFT calibration case, if one particular calculation of $F_j(\xb)$ is performed, then there are several additional $F_{j'}(\xb)$ that one ought to calculate to avoid a loss in efficiency, for example when $F_j$ and $F_{j'}$ are observables of the same nucleus. 
We can propose several means to extend the methodology presented in this paper to handle these more realistic computational constraints, but in order not to distract from the key ideas being presented, we elect to make the simplifying assumption of computational independence of all $\{F_j(\xb)\}$ calculations. 

\subsection{Stochastic Average Model Methods}
In recent work, we proposed a framework for a class of methods for this problem setting which we call \emph{stochastic average model (SAM) methods} \cite{SAMP2022}. 
SAM methods take inspiration from stochastic average gradient (SAG) methods \cite{Schmidt2013} and, more specifically, the SAGA method \cite{DefazioBL14} by employing importance sampling via an \emph{ameliorated model}. 

The key idea behind SAM methods is that they maintain, for each of $j=1,\dots,p$, a separate model $m_j(\xb; \cb_j)\approx F_j(\xb)$, intended to be valid for all $\xb$ in a neighborhood of $\cb_j\in\Reals^{dim}$. 
We omit technical details that quantify approximation quality, but as a pedagogical example, one could imagine that $m_j(\xb; \cb_j)$ is a(n approximate) first-order Taylor model of $F_j$ centered at $\cb_j$; that is, 
$m_j(\xb; \cb_j) = F_j(\cb_j) + \nabla F_j(\cb_j)^\top (\xb - \cb_j)$.
Being iterative methods, SAM methods generate a sequence of \emph{incumbents} $\{\xb^k\}$ that serve as current best guesses of a solution to \cref{eq:sum} in the $k$th iteration of the method. 
At the start of the $k$th iteration of a SAM method, only a subset $I^k\subseteq\{1,\dots,p\}$ indexing center points $\{\cb_j: j\in I^k\}$ are updated to the incumbent $\xb^k$. 
In the specific example of using first-order Taylor models as each $m_j(\xb;\cb_j)$, such an update entails the computation of gradients $\{\nabla F_j(\xb^k): j\in I^k\}$;
note, importantly, that we do not compute $\{\nabla F_j(\xb^k): j\not\in I^k\}$ in this pedagogical example, potentially yielding savings in expensive gradient computations.
In black-box settings where gradient computations are either very difficult or impossible to perform, one can employ model-based derivative-free optimization (DFO), see e.g. \cite[Section 2.2]{LMW2019AN}.
In model-based DFO, the update of a model center $\cb_j$ to an incumbent $\xb^k$ entails performing enough additional function evaluations of $F_j$ to guarantee that a(n interpolation) model $m_j(\xb; \xb^k)$ is a \emph{fully linear} model of $F_j$ in a neighborhood (of some prescribed volume) of $\xb^k$; see e.g. \cite[Chapter 6]{Conn2009a} for a detailed discussion of fully linearity and example numerical algorithms for constructing such models.  

Having defined $p$ independent models $\{m_j: j=1,\dots,p\}$ with corresponding centers $\cb_j$ (we will denote centers as $\cb^k_j$ when discussing them within the context of an iterative method), we can then define the \emph{average model} in the $k$th iteration,
\begin{equation}
    \label{eq:average_model}
    \bar{m}^k(\xb) := \displaystyle\sum_{j=1}^p m_j(\xb; \cb_j^k). 
\end{equation}
One could propose, at this point, an algorithm that iteratively solves a sequence of (trust-region) subproblems, each minimizing an average model of the form in \cref{eq:average_model}. 
However, analogously to SAG methods, the average model $\bar{m}^k$ clearly exhibits a good deal of bias as a pointwise estimate of the \emph{ideal model} in the $k$th iteration, 
\begin{equation}
    \label{eq:ideal_model}
    m^k(\xb) = \displaystyle\sum_{j=1}^p m_j(\xb; \xb^k) 
\end{equation}
in the sense that the expectation, with respect  to the filtration generated by $I^0,I^1,\dots, I^k$, of the average model $\bar{m}^k$ evaluated at $\xb$ can be a poor estimate of the ideal model value $m^k(\xb)$. 
With this bias in mind, we construct instead the \emph{ameliorated model}, 
\begin{equation}
    \label{eq:ameliorated_model}
    \hat{m}^k(\xb) := \bar{m}^{k-1}(\xb) + \displaystyle\sum_{j\in I^k} \frac{m_j(\xb;\xb^k)-m_j(\xb;\cb_j^{k-1})}{\pi^k_j},
\end{equation}
where $\pi^k_j$ denotes the probability
    $\pi^k_j := \Prob\left[ j\in I^k\right]$. 
We can immediately demonstrate just from definitions that the ameliorated model is a pointwise unbiased estimate of the ideal model: 
\begin{theorem}\label{thm:unbiased}
For any $\xb\in\Reals^{dim}$, 
    $$\mathbb{E}_{I^k}\left[ \hat{m}^k(\xb) \right] = m^k(\xb).$$
\end{theorem}

\begin{proof}{Proof}
$$\begin{array}{rl}
\Exp_{I^k}\left[ \hat{m}^k(\xb)\right] = & 
\Exp_{I^k}\left[\bar{m}^{k-1}(\xb) + \displaystyle\sum_{j=1}^p \frac{m_j(\xb;\xb^k)-m_j(\xb;\cb_j^{k-1})}{\pi^k_j}  \Ind\left[j\in I^k\right]\right]\\
= & \bar{m}^{k-1}(\xb) + \displaystyle\sum_{j=1}^p \left(m_j(\xb;\xb^k)-m_j(\xb;\cb_j^{k-1})\right) \\
= & \displaystyle\sum_{j=1}^p \left(m_j(\xb;\xb^k)-m_j(\xb;\cb_j^{k-1}) + m_j(\xb;\cb_j^{k-1})\right)\\
= & \displaystyle\sum_{j=1}^p m_j(\xb;\xb^k) = m^k(\xb). \\
\end{array}
$$
\end{proof}
With this unbiasedness made precise, we see that in each iteration $k$, \emph{any} choice of nonzero probabilities $\{\pi^k_j: \pi^k_j \in (0,1]\}$ will yield an unbiased ameliorated model. 

With the definition of the ameliorated model $\hat{m}^k$ in hand, we can describe a general SAM method in \Cref{alg:sam}.

\begin{algorithm}[h!]
\caption{Stochastic Average Model Method}
\label{alg:sam}
\DontPrintSemicolon

\KwIn{Initial point $\xb^0\in\Reals^{dim}$, maximum trust-region radius $\Deltamax > 0$, initial radius $\Delta_0\in(0,\Deltamax)$, increase factor $\gammainc > 1$, success constant $\eta_1\in(0,1)$, growth factor $\eta_2>0$, expected batch size $b\in[1,p]$}

Initialize model centers $\{\cb^j \gets \xb^0: j=1,\dots,p\}$\;
Construct models $\{m_j(\xb; \cb^j): j=1,\dots p\}$\;

\For{$k=1,2,\dots,K$}{
    Choose $\pi^k \in b\DD^p$ and draw $I_k$ as an independent sample according to $\pi^k$\; \label{line:ask_for_probs}
    
    \ForEach{$j\in I_k$}{
        $\cb^k_j\gets \xb^k$\; \label{line:model_updates}
    }
    \ForEach{$j\in\{1,\dots,p\}\setminus I_k$}{
        $\cb^k_j\gets \cb^k_{j-1}$\;
    }
    
    (Approximately) solve the trust-region subproblem:
    \begin{equation}\label{eq:trsp}
    \sv^k \gets \arg\min_{\sv:\|\sv\|\leq\Delta_k} \hat{m}^k(\xb^k + \sv)
    \end{equation}\; \label{line:trsp}
    
    Compute $f(\xb^k + \sv^k)$ and
    \[
    \rho_k \gets \frac{f(\xb^k) - f(\xb^k+\sv^k)}{\hat{m}^k(\xb^k) - \hat{m}^k(\xb^k + \sv^k)}
    \] \label{line:rho_compute}
    
    \eIf{$\rho_k > \eta_1$ \textbf{and} $\Delta_k < \eta_2\|\nabla \hat{m}^k(\xb^k)\|$\label{line:update_tr}}{
        $\xb^{k+1} \gets \xb^k + \sv^k$\;
        $\Delta_{k+1}\gets\min\{\gammainc\Delta_k, \Deltamax\}$\;
    }{
        $\xb^{k+1} \gets \xb^k$\;
        $\Delta_{k+1}\gets \gammainc^{-1}\Delta_k$\;
    }
}
\end{algorithm}

In the initialization step of \Cref{alg:sam}, we initialize all model centers $\cb^k_j$ to $\xb^0$ and construct a set of models $m_j(\xb;\cb^j)$; we note that a prior model could alternatively be provided in this step if it were available. 
At the start of the $k$th iteration, a scaled discrete probability distribution $\pi^k$ on $\{1,\dots,p\}$ is generated and a random sample $I_k$ is drawn such that $\mathbb{P}[j\in I_k] = \pi^k_j$ for all $j=1,\dots,p$.
Much more discussion will come in this paper regarding the generation of $\pi^k$. 
Throughout this paper we denote the simplex $\DD^p:= \{\zb\in\Reals^p: \sum_{j=1}^p \zb_j = 1, \zb_j\geq 0 \hspace{1pc} \forall j=1,\dots,p\}$.
As a practical remark, the requirement that $\pi^k\in b\DD^p$ only ensures that the expectation of $|I_k|$ is $b$; it does not enforce that $|I_k|=b$ deterministically. 
This will be discussed in more detail in \Cref{sec:implementation}.
The ameliorated model is then computed according to \cref{eq:ameliorated_model} using the data $\pi^k$ and $I_k$. 
The remainder of the iteration of the SAM method is identical to a standard trust-region optimization method, with an ameliorated model being used as the objective model in each iteration.  
Convergence results for SAM methods have been previously established \cite{SAMP2022}. 

As a critical remark, for the sake of exposition, the SAM method in \Cref{alg:sam} is slightly different from the SAM method presented in \cite{SAMP2022}.
In particular, \Cref{alg:sam} only employs an ameliorated model for the purpose of constructing a model of the objective function to be minimized in the trust-region subproblem.
As stated, \Cref{alg:sam} does not exhibit an independent ameliorated model (indexed by a secondary set $J_k$ in \cite{SAMP2022}) to estimate the decrease in the objective function in the numerator of $\rho_k$.
This choice implies that the sequence of incumbents $\{\xb^k\}$ generated by \Cref{alg:sam} correspond to a sequence of objective values $\{f(\xb^k)\}$ that are monotonically nonincreasing, due to the definition of $\rho_k$. 
However, this choice is entirely expositional, made for the sake of clearly explaining our use of contextual bandits.
Our implementation of \Cref{alg:sam} discussed in \Cref{sec:implementation} does in fact generate a second set $J_k$ corresponding to a separate set of probabilities $\pi^{k, J}$ generally distinct from $\pi^k$.  
These details will be discussed in \Cref{sec:implementation}. 

Intuitively, some probabilities $\{\pi^k_j\}$ will yield realizations of the ameliorated model that are empirically better than others. 
In previous work, 
we proposed choosing a set of probabilities $\{\pi^k_j\}$ that would yield a minimum variance unbiased estimator in a precise sense, inspired by work including \cite{CsibaRichtarik2018,HorvathRichtarik2019}.  
This was accomplished by first restricting the sampling of $I_k$ to independent (Poisson) sampling, as is shown in 
\lineofalg{line:ask_for_probs}{alg:sam}. 
We then formulated the optimization problem that minimized the variance of the ameliorated model under Poisson sampling subject to a constraint on the expected cardinality of $|I^k|$ ($|I^k|\leq b$), as well as bound constraints on each $\pi^k_j$ ($0\leq \pi^k_j \leq 1$ for all $j=1,\dots,p$). 
Analysis of the Karush Kuhn Tucker conditions reveals that the solution to this optimization problem can be stated in closed form as
\begin{equation}
\label{eq:opt_probs}
\pi_{(j)}^{k*} = \left\{ 
\begin{array}{rl}
(b + c - p) \displaystyle\frac{d^k_{(j)}}{\displaystyle\sum_{i=1}^c d^k_{(i)}} &\text{ if } j\leq c\\
1 & \text{ if } j > c,\\
\end{array}\right.
,
\end{equation}
where $c$ is the largest integer satisfying 
\begin{equation*}
\label{eq:kkt_ineqs}
0<b+c-p\leq \displaystyle\sum_{i=1}^c \displaystyle\frac{d^k_{(i)}}{d^k_{(c)}},
\end{equation*}
we define
\begin{equation}\label{eq:discrepancy}
d^k_j :=  |m_j(\xb;\xb^k) - m_j(\xb;\cb_j^{k-1})|,
\end{equation}
 and we have used the order statistics notation
$d^k_{(1)} \leq d^k_{(2)}\leq \dots \leq d^k_{(p)}.$ 

Obviously, the discrepancies $\{d^k_j\}$ viz. \cref{eq:discrepancy} are unknowable.
If we knew the values of $\{m_j(\xb;\xb^k): j=1,\dots,p\}$, then we could just minimize the ideal model $m^k$ in every iteration of \Cref{alg:sam}, which would be practically preferable to using the randomized ameliorated model $\hat{m}_k$. 

In our previous work \cite{SAMP2022}, we considered different classes of models $m_j$ and used upper bounds on model error to coarsely approximate the value of the discrepancies $d_j^k$. 
However, for virtually any imaginable class of models with practical value in smooth optimization, these upper bounds in turn depend on unknowable constants such as Lipschitz constants of (higher-order) derivatives. 
This is certainly the case, in particular, for both first-order Taylor models and for fully linear models in the DFO setting.
Thus, we proposed simply approximating these Lipschitz constants either via estimates of higher-order derivatives or by the largest observed magnitude of recent (model) gradient/displacement secants. 
More details concerning this Lipschitz estimation will be provided in \Cref{sec:lipschitz}. 

However, employing either such Lipschitz estimation technique to estimate discrepancies $d^k_j$ is undesirable.
As an extreme, but not practically impossible, example, if an estimate of a Lipschitz constant for $\nabla F_j$ is large relative to estimates of all other Lipschitz constants for $\nabla F_{j'}$ with $j'\neq j$ in an early iteration of a SAM method, then the probability assignment in \cref{eq:opt_probs} will proportionally oversample $j$ in subsequent iterations. 
In other words, choosing $\{\pi^k_j\}$ according to \cref{eq:opt_probs} with $\{d^k_j\}$ estimated in this fashion can result in a heavy emphasis on \emph{exploitation} while neglecting \emph{exploration}.

We illustrate this type of failure in \Cref{fig:failure}. 
In this toy example, we let $p=4$ and we let each component function be $F_j(\xb) := 10^j(\xb_j - j)^2.$ We set the batch size $b=1$. 
To create something of a strawman, we changed a few lines of the implementation of the method tested in our numerical results so that it \emph{does not employ the bandit method} that we will propose in what follows. 
This means the strawman method is nearly the \texttt{SAM-POUNDERS} method described in \cite{SAMP2022}, with model updates handled by model-based DFO techniques; in particular (as discussed in \Cref{sec:implementation}), we employ quadratic interpolation models. 
By our choice of a quadratic $F(\xb)$ in this toy example, it is clear that a quadratic interpolation model $m_j$ will \emph{exactly equal} $F_j$ within a finite number of function evaluations. 
Moreover, if we only employ the Lipschitz estimation scheme discussed in more detail in \Cref{sec:lipschitz} to generate $\pi^k$ in \lineofalg{line:ask_for_probs}{alg:sam}, then this also means that eventually the gradient Lipschitz constant estimate for $F_j$ will be exactly $2\times 10^j$.
However, by the same estimation scheme, every gradient Lipschitz constant estimate is initialized to $0$; thus, in effect, the first model $m_j$ to collect enough samples to detect the nonzero curvature in the $j$th direction of $F_j$ will effectively be sampled deterministically. 
Precisely this failure is seen in the realization shown in \Cref{fig:failure}, where apparently $m_2$ was the first model under the Lipschitz estimation scheme to detect nonzero curvature. 

\begin{figure}
    \centering
    \includegraphics[width=.99\textwidth]{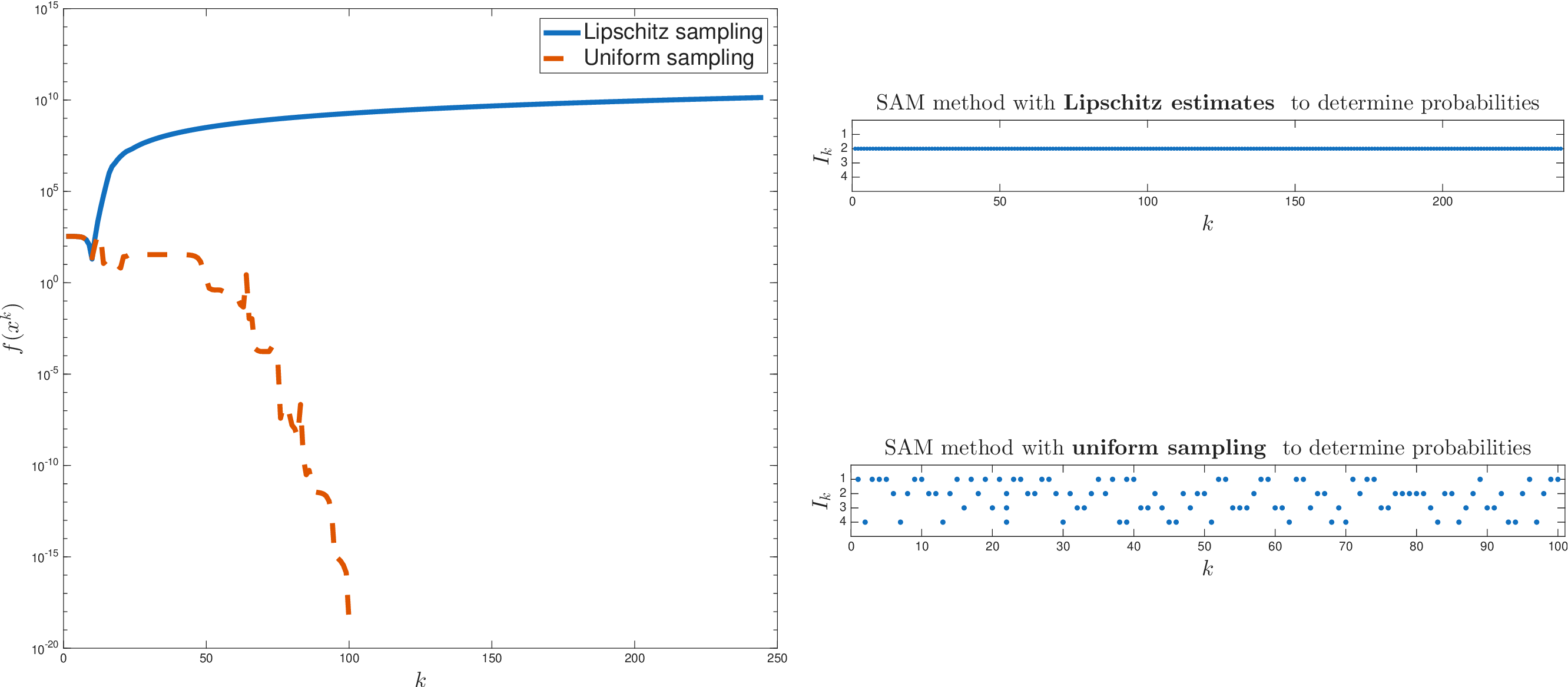}
    \caption{\label{fig:failure} \texttt{SAM-POUNDERS} \cite{SAMP2022} employing quadratic interpolation models applied to a synthetic quadratic problem described in the text. The first variant uses the Lipschitz estimation scheme described in \Cref{sec:lipschitz} to generate $\pi^k$, while the latter updates models uniformly at random. \textbf{Left: } A comparison of a typical realization of the two variants comparing iteration counter to true function value $f(\xb^k)$ (unknown to the randomized method). We see that the method employing the Lipschitz estimation scheme for sampling diverges, while the method employing uniform sampling converges. \textbf{Right: } A visualization of which model was updated in each iteration of both methods. As expected, \texttt{SAM-POUNDERS} employing only Lipschitz estimation gets stuck updating $m_2$ and has no means to recover.}
\end{figure}
In the same \Cref{fig:failure}, we contrast this failure with the effect of simply performing \emph{uniform random sampling}, i.e. $\pi^k = \left[0.25, 0.25, 0.25, 0.25\right]$ for all $k$. 
In this case, we see that had we simply committed to this mode of pure exploration instead of attempting something more exploitative (Lipschitz estimation), the algorithm performance would have been vastly better.

\subsection{Our Contribution}
The purpose of the work discussed in this manuscript is to address precisely the issues resulting from a exploration-exploitation tradeoff seen in \Cref{fig:failure} in a disciplined way by adopting techniques of contextual bandit learning.
This results in running, in parallel with a SAM method, a very minor adjustment to a well-known algorithm for non-stochastic contextual bandits, Exp4 \cite{auer2002nonstochastic}, to generate probability distributions $\pi^k$.
This combination of SAM methods with the variant of Exp4 results in a
practical and theoretically satisfying solution to address the imbalance in exploitation and exploration that existed in our past work.

It is important to note that the convergence results of our past work are not invalidated by the novel work presented here. 
Our past convergence analysis \cite{SAMP2022} was based on the almost sure convergence analysis employed in the STORM framework \cite{Chen2017, BCMS2018}. 
As such, it was assumed that $\pi^k$ was chosen in some way to guarantee that $\{\hat{m}^k\}$ were probabilistically fully linear models. 
In the implementation that accompanied that paper, a \emph{dynamic} batch size $b_k$ was employed that grew with, among other quantities, the inverse of the trust-region radius $\Delta_k$.
As a result, it was virtually always the case that $b_k=p$ by the end of the optimization, thereby recovering deterministically fully linear models, and hence recovering the dynamics of a deterministic trust-region method. 
In this paper, motivated by mitigating computational expense and not over-eagerly increasing $b_k$ based on conservative estimates, we sought (and managed) to fix a \emph{constant} batch size (e.g., $b=1$) in our numerical experiments, highlighting another practical utility of adopting a bandit framework for generating $\pi^k$ in a SAM method. 

Finally, the flexibility of the contextual bandit setting opens other interesting opportunities for methods of generating $\pi^k$ besides attempting to approximate $\pi^{k*}$ in \cref{eq:opt_probs}. 
We specifically identify two classes of these alternatives, which have practical implications in scientific workflows.
The first alternative is in settings where lower-fidelity models or pre-trained surrogates exist for each $F_j$. 
The second takes full advantage of the generality of contextual bandit learning and explicitly allows a domain expert (or foundation model) to inject their beliefs via an arbitrary declaration of $\pi^k$. 

\subsection{Related Work}
The literature on bandit methods is vast and already spans decades. 
The multi-armed bandit (MAB) problem, introduced in \cite{Robbins1952}, provides a canonical framework for studying the exploration--exploitation trade-off in sequential decision making.
In a MAB problem, a (finite) number of bandits\footnote{``One-armed bandit" is a mostly archaic American slang term for a slot machine.} are available to a player; the rewards from pulling an arm are assumed to be governed by some (fixed) distribution, unknown in advance to the player. 
Intuitively, the player's objective is to maximize their total rewards given a fixed budget of sequential arm pulls. 
In the now-classical analysis of bandit problems, one specifies a policy for the player to follow, and then defines a notion of regret with respect to a reference policy. 
Foundational asymptotic lower bounds on regret for the stochastic MAB model were established in \cite{LaiRobbins1985, Auer2002}. 

In contrast, adversarial (non-stochastic) bandits, formalized by \cite{auer2002nonstochastic} and further synthesized in \cite{BubeckCesaBianchi2012}, abandon distributional assumptions and instead guarantee sublinear regret against \emph{arbitrary} sequences of rewards. 
This is the perspective that we will adopt in this paper since, as we will see, our rewards are nonstationary and are effectively a deterministic function of the state of the SAM method. 
We note that while adversarial models are sometimes used as a proxy for nonstationarity, this is a coarse worst-case abstraction that does not explicitly model temporal structure. 
Frameworks do exist that better capture realistic drift while retaining statistical efficiency, see \cite{BubeckSlivkins2012}. 
While developing better models of drift that depend on the partial observability of the state of a SAM method is an avenue for future work, we see the adversarial lens as a useful and meaningful first step. 

A natural extension to the adversarial setting is the contextual bandit setting, where side information is available at each round and the player can condition their action on observed covariates.
Early formulations appeared in \cite{LangfordZhang2008, Li2010}; however, the Exp4 method employed in this work was introduced in \cite{auer2002nonstochastic} and can in fact be viewed as a contextual bandit method where the side information is presented in the form of experts. More generally, contextual bandits connect closely to supervised learning via reductions and policy optimization, as surveyed in \cite{Agarwal2014}, and form a bridge between classical bandits and reinforcement learning.

Finally, we note that this is not the first work to consider employing bandit methods to assist in sampling for randomized methods for solving \cref{eq:sum}.  
Two works independently derived sampling distributions for variance reduction within the context of stochastic gradient descent and SAGA \cite{namkoong2017adaptive, salehi2017stochastic} using an algorithm for stochastic MAB, Exp3, which was also introduced in \cite{auer2002nonstochastic}. 
The use of Exp3 was also investigated as a generator of sampling distributions within Adam \cite{kingma2014adam} in \cite{liu2020adam}. 
\cite{borsos2018} employed a bandit method, Follow the Regularized Leader, introduced in \cite{abernethy2009competing}, also for variance reduction in stochastic gradient descent.
By using a MAB model, these works effectively all assumed a stationary reward distribution. 
Some efforts were made to explicitly address nonstationarity in later works \cite{el2020adaptive, zhao2025} by adapting bandit methods to an assumption of nonstationary reward distributions. 

We note that the union of all these existing works that employ bandit methods to sample in finite sum minimization problems were motivated primarily by learning problems, especially deep learning and federated learning, and thus it was of paramount importance that the cost of a bandit algorithm not dominate the (fairly inexpensive) cost of computing a stochastic gradient estimator. 
Because this work is motivated by far more expensive computational settings, we are able to employ relatively expensive ``experts" in Exp4 like pretrained surrogates, low-fidelity models, or even human/AI expertise; it is because of this niche that this work occupies that we believe this particular work is novel. 

In the remainder of this paper, we introduce the setting of contextual bandit learning in \Cref{sec:contextual}.
We specify the variant of Exp4 that we employ in \Cref{sec:exp4}. 
We propose alternative methods for generating $\pi^k$ in \Cref{sec:experts}. 
We perform preliminary numerical experiments in \Cref{sec:experiments} and provide concluding remarks in \Cref{sec:conclusion}.

\section{Contextual Bandit Learning}\label{sec:contextual}
In this section, will deliberately overload notation in our definition of the contextual bandit learning problem in an effort to make it clear how we choose to cast the problem of sampling in a SAM method as a contextual bandits problem. 
\Cref{tab:notation_guide} will explicitly clarify the use of overloaded notation. 

A \emph{bandit problem} involves both a player and an environment participating in a game of a discrete number $K$ of rounds, which we index by $k$. 
In the $k$th round of the game, the player is presented with $p$ ``arms".
The $j$th arm is assigned a reward $d_j^k$, but the player does not know any of the rewards at the start of the round. 
The player must choose $b$ many arms in the $k$th round; we will denote her selection by $I^k\subset\{1,\dots,p\}$ such that $|I^k|=b$. 
The player then immediately observes each $\{d_j^k: j\in I^k\}$ and the $k$th round ends. 

What distinguishes a typical bandit problem from a \emph{contextual bandit problem with experts} is the additional presence of entities usually called \emph{experts}. 
At the beginning of the $k$th round, the player is additionally provided a set of $N$ probability vectors $\mathbf{e}^{k,1}, \dots, \mathbf{e}^{k,N}\in\DD^p$. 
We refer to each such probability vector $\mathbf{e}^{k,n}\in\DD^p$ as the $n$th expert's \emph{advice}.
The player does not decide which arms to pull until after she receives advice from all $N$ experts.
For the sake of development, how the experts generate their advice is irrelevant and should be treated as a black box.
For a more glass-box treatment, the experts can be viewed as a set of (supervised learning) models that ingest some quantifiable and measurable \emph{context} produced by the environment, and then output advice.

In the bandit problem with experts, a standard notion of regret is defined with respect to consistently following a fixed expert's advice. 
That is, if the player consistently heeded the $n$th expert's advice, and then randomized her draws so that $\Prob[j\in I_k] = \mathbf{e}^{k,n}_j$ for $j=1,\dots,p$, then the expected cumulative reward that she will receive at the end of $K$ rounds is
$$C^{K,n} := \displaystyle\sum_{k=1}^K \displaystyle\sum_{j=1}^p \mathbf{e}^{k,n}_j d^k_j.$$
Thus, a good benchmark for any algorithm the player devises is to compare her algorithm's expected cumulative reward to the expected cumulative reward of the \emph{best} expert. 
That is, we define the player's \emph{expected regret} from using a fixed algorithm to play this game for $K$ rounds as
\begin{equation}\label{eq:regret}
R^K := \displaystyle\max_{n=1,\dots,N} C^{K,n} - \mathbb{E}\left[\displaystyle\sum_{k=1}^K \displaystyle\sum_{j\in I_k} \mathbf{d}_j^k\right],
\end{equation}
where the expectation is taken over the filtration defined by the random variables $\{I_k: k=1,\dots,K\}$. 

If the player employs a canonical algorithm known as \textbf{Exp4} (\textbf{Exp}onential-weight algorithm for \textbf{Exp}loration and \textbf{Exp}loitation using \textbf{Exp}ert advice) \cite{auer2002nonstochastic} to determine each $\Prob[j\in I_k]$, then one can show her regret is bounded like $R^K\in\mathcal{O}\left(\displaystyle\sqrt{p\log(N)K} \right)$, 
provided we make the simplifying assumption that the rewards are scaled so that $d^k_j\in[0,1]$ for all $k,j$. 
That is, given the particular notion of regret $R^K$ in \cref{eq:regret}, \textbf{Exp4} yields regret sublinear in $K$.
Sublinear regret is desirable in bandit problems since it implies that the player's average regret $R^K/K\to 0$ as the rounds of the game $K\to\infty$. 

By the choice of our overlapping notation, it is hopefully already clear how we choose to map the problem of choosing sampling distributions in a SAM method onto a bandit problem with experts. 
We highlight this overlap in \Cref{tab:notation_guide} for completeness. 
In summary, at the end of the $k$th iteration (round) of a SAM method, the player must choose a subset $I_k$ of model centers to update to the next iteration's incumbent, $\xb^{k+1}$, in order to define the ameliorated model $\hat{m}^k$, see \cref{eq:ameliorated_model}.
Because the update of models $\{m_j:j\in I_k\}$ entails new function evaluations at points including $\xb^{k+1}$, the player receives rewards $d^k_j$ in the form of absolute model errors, see \cref{eq:discrepancy}. 
These rewards (error magnitudes) directly inform the player about how useful the model update she chose to perform was, which allows her to accordingly update her algorithmic strategy (in this paper, one determined by Exp4) for performing future updates. 

\begin{table}
\begin{tabular}{@{}l@{\quad}l@{\quad}l@{}}
\hline 
Notation              & Bandit Learning     & SAM method            \\ \hline
$k$    & Index for rounds of play   & Iteration counter \\ 
$p$   & Finite number of arms   & Finite number of component functions $F_j$ \\
$d_j^k$    & Reward for pulling arm $j$ in round $k$ & Error of $m_j$ in $k$th iteration (see \cref{eq:discrepancy}) \\
$I_k$     & Subset of arms in $\{1,\dots,p\}$ pulled in round $k$ & Subset of $\{m_j: j = 1,\dots,p\}$ recentered at $\xb^k$\\ \hline
\end{tabular}
\caption{Notation guide, highlighting the intentional overlap in notation between Exp4 (\Cref{alg:exp4}) and SAM methods (\Cref{alg:sam}). \label{tab:notation_guide}}
\end{table}

\section{A Bandit Method, Exp4}\label{sec:exp4}
Given our setting and notation, we state a variant of Exp4 in \Cref{alg:exp4}. 
In Exp4, the player maintains a set of weights $\{w_{k,n}\}$, one for each player $n\in \{1,\dots,N\}$,  updated in each round $k$.
If the player had a prior relative weighting that they would place on expert advice, then $w_{1,n}$ could be instantiated with those weights; by default, we simply initialize $w_{1,n}=1$ for all $n$ in \Cref{alg:exp4}, thereby giving all experts equal weight.
In the $k$th round, the player receives the advice $\eb^{k,n}\in b\DD^p$ from each expert. 
The probability distribution $\pi^k_j$ is generated by first taking the normalized weighted sum over the $n$ advices, and then weighting this result with the uniform distribution over all $p$ bandits, rescaled to a vector in $b\DD^p$. 
Then, $I_k$ is drawn such that $|I_k|=b$ and $\Prob[j\in I_k] = \pi^k_j$ for all $j=1,\dots,p$.
The corresponding rewards $\{d^k_j: j\in I_k\}$ are received and are then scaled by their corresponding probabilities $\pi^k_j$. 
Then, every expert's weight is updated by a multiplicative factor that is exponential in the inner product between the corresponding expert's advice and the observed scaled reward, and the round ends. 
It is the exponential magnitude of this update that contributes to one of the ``Exp"s in Exp4's name, and it becomes apparent in analysis (see the proof of \Cref{thm:sublinear_regret}) why an exponential update is necessary. 
It is instructive to note the intuition that an alignment of an expert's relative confidence in their advice with the bandits that actually produced large rewards results in a proportionally large update in the weight that the player will assign to that expert in future rounds. 


\begin{algorithm}[h!]
\caption{Exp4}
\label{alg:exp4}
\DontPrintSemicolon

\KwIn{Exploration parameter $\gamma\in(0,1]$, batch size $b < p$}

Initialize expert weights $w_{1,n}\gets 1$ for $n=1,\dots,N$\;

\For{$k=1,2,\dots,K$}{
    Get expert advice vectors $\{\mathbf{e}^{k,n}: n=1,\dots,N\}$\;
    
    $W_k\gets \sum_{n=1}^N w_{k,n}$\;
    
    $\pi_j^k\gets (1-\gamma)\sum_{n=1}^N \frac{w_{k,n} e_j^{k,n}}{W_k} + \frac{\gamma b}{p}$ for $j=1,\dots,p$\; \label{line:probs_out_of_exp4}
    
    Draw $I_k\subset \{1,\dots,p\}$ such that $\Prob[j\in I_k] = \pi_j^k$ for $j=1,\dots,p$\;\label{line:draw_ik}
    
    Receive rewards $\{d_j^k: j\in I_k\}$\; \label{line:ask_for_reward}
    
    Set $\hat{d}_j^k \gets 
    \begin{cases}
    d_j^k/\pi_j^k & \text{if } j\in I_k \\
    0 & \text{otherwise}
    \end{cases}
    \quad \text{for } j=1,\dots,p$\;
    
    $w_{k+1,n}\gets w_{k,n}\exp\left(\frac{\gamma \mathbf{e}^{k,n\top} \hat{\mathbf{d}}^k}{p}\right)$ for $n=1,\dots,N$\; \label{line:update_weights}
}
\end{algorithm}
As previously suggested, we will demonstrate that \Cref{alg:exp4} exhibits sublinear regret with respect to a best expert. 
By itself, this result is not completely new, and is just a slight modification of results proven in \cite{auer2002nonstochastic}. 
However, we provide a proof in the appendix for completeness since \Cref{alg:exp4} differs slightly from statements found in the bandit literature, both in terms of notation and the use of a batch size parameter $b>1$. 
The first result we demonstrate bounds regret $R^K$ in terms of an arbitrary exploration parameter $\gamma$. 
Define an \emph{uniform expert} as an expert that, given $p$ arms, chooses $\mathbf{e}^{k}_j = \frac{1}{p}$ for all $j\in\{1,\dots,p\}$ and for all $k$. 

\begin{theorem}\label{thm:sublinear_regret}
Suppose one of the $N$ experts providing advice in \Cref{alg:exp4} is an uniform expert.
Suppose there is a uniform bound over all rewards $\mathbf{d}^k_j\leq D$ for all $j\in\{1,\dots,p\}$ and $k\in\{1,\dots,K\}$. 
Then, the player's expected regret $R^K$, defined in \cref{eq:regret}, is bounded like
$$R^K\leq ((e-2)bD + 1)\gamma \displaystyle\max_{n=1,\dots,N} C^{K,n} + \displaystyle\frac{p \ln N}{\gamma}.$$
\end{theorem}

As an immediate corollary, by choosing $\gamma$ as a function of problem data, we can show that $R^K$ is sublinear in $K$.  

\begin{corollary}\label{cor:sublinear}
    Let 
    $$\gamma = \displaystyle\sqrt{\frac{p\ln(N)}{((e-2)bD + 1)bDK}}.$$
    Then, 
    $R^K=\mathcal{O}\left(bD\sqrt{p\ln(N)K} \right).$
\end{corollary}

\begin{proof}
As a very coarse upper bound, $C^{K,n} \leq bDK$ for all $n\in\{1,\dots,N\}$. The result follows. 
\end{proof}

We will employ \Cref{cor:sublinear} to inform the setting of the parameter $\gamma$ in our numerical experiments, and will make additional comments in \Cref{sec:experts}. 

\section{Examples of Experts}\label{sec:experts}

From \Cref{cor:sublinear}, our motivation is now clear. 
We will generate a collection of experts $\{\eb^n:n\in 1,\dots,N\}$ and mix their per-iteration advice $\eb^{k,n}$ according to \Cref{alg:exp4} in order to generate distributions $\pi^k$ for use in \Cref{alg:sam}. 
We would like to suppose that at least one expert $\eb^n$ advises, for all $k$, a probability distribution $\eb^{k,n}$ that yields sufficient variance reduction in the construction of the ameliorated model $\hat{m}_k$ to make $\hat{m}_k$ a suitable (for example, probabilistically fully linear \cite{Chen2017, BCMS2018}) model of the objective $f$ in the trust region defined by $\xb^k$ and $\Delta_k$. 
Ideally, such an expert $\eb^n$ always approximates the optimal, but unknowable, distribution \cref{eq:opt_probs}.
Practically, however, only an expert with essentially perfect oracular knowledge of $f$ could satisfy any such guarantees.
Thus, being practically motivated, we strive to design good experts, but fall back on our assurance that, as demonstrated in \Cref{cor:sublinear}, the performance of $\pi^k$ generated by \Cref{alg:exp4} can never be significantly worse than uniform random sampling, in a precisely quantifiable sense. 

The subsections following this discussion focus on how to design these experts, but for concreteness, we first show explicitly the connections between \Cref{alg:sam} and \Cref{alg:exp4}, graphically depicted in \Cref{fig:tikz}.
\Cref{alg:sam} and \Cref{alg:exp4} are simultaneously initialized and share an iteration counter $k$. 
In \lineofalg{line:ask_for_probs}{alg:sam}, the SAM method will wait for \Cref{alg:exp4} to return $\pi^k\in b\DD^p$, computed in \lineofalg{line:probs_out_of_exp4}{alg:exp4}. 
Using the probabilities $\pi^k$, a realization of $I_k$ is generated in \lineofalg{line:draw_ik}{alg:exp4}; the model centers corresponding to $I_k$ are then updated in \lineofalg{line:model_updates}{alg:sam}. 
With these model updates, one can then (approximately) compute, for each $j\in I_k$,
\begin{equation}\label{eq:max_discrepancy_in_delta_ball}
d^k_j := \displaystyle\max_{\sv: \|\sv\|\leq\Delta_k} |m_j(\xb^k + \sv;\xb^k) - m_j(\xb^k + \sv; \cb_j^{k-1})|.
\end{equation}
We note that the tractability of solving \cref{eq:max_discrepancy_in_delta_ball} to global optimality is, of course, dependent on the choice of model class to which $m_j$ belongs.
By design, the rewards in \cref{eq:max_discrepancy_in_delta_ball} are directly related to the unknowable discrepancies \cref{eq:discrepancy} that appear in the variance of the ameliorated model. 
In this sense, per-iteration maximization of rewards is directly related to per-iteration maximization of variance reduction. 
The reward $d^k_j$ is fed into \lineofalg{line:ask_for_reward}{alg:exp4} and the per-expert weights are eventually updated with these rewards in \lineofalg{line:update_weights}{alg:exp4}. 

\begin{figure}\caption{\label{fig:tikz} An illustration of a single iteration $k$ of a SAM method (\Cref{alg:sam}) and the variant of Exp4 (\Cref{alg:exp4}) being performed in parallel.
The oracle that computes $F(\xb)$ is visualized as existing outside of the environment that contains both the SAM method and Exp4, and is only capable of communicating with the SAM method. 
We note that the experts live in the environment that contains both the SAM method and Exp4 and is allowed to observe the dynamics of the SAM method. 
However, the experts are incapable of interacting directly with the $F(\xb)$ oracle.}
\begin{tikzpicture}[
>=Stealth,
node distance=1.7cm,
process/.style={
    rectangle,
    draw,
    rounded corners,
    align=center,
    text width=2.9cm,
    minimum height=0.9cm
},
plainbox/.style={
    rectangle,
    draw,
    align=center,
    text width=2.6cm,
    minimum height=0.9cm
},
group/.style={
    rectangle,
    draw,
    rounded corners,
    inner sep=6pt
},
arrow/.style={->,thick},
bidir/.style={<->,thick}
]


\node[process, anchor=north] (a1model)
{Update model centers $\{c_j^k: j\in I_k\}$\\ (Line \ref{line:model_updates})};

\node[process,below=of a1model] (a1trsp)
{Solve subproblem \cref{eq:trsp}\\ (Line \ref{line:trsp})};

\node[process,below=of a1trsp] (a1evaluate)
{Evaluate $f(x^k+s^*)$\\ (Line \ref{line:rho_compute})};

\node[process,below=of a1evaluate] (a1update)
{Update $x^k$ and $\Delta_k$\\ (Line \ref{line:update_tr})};

\draw[arrow] (a1model) -- (a1trsp);
\draw[arrow] (a1trsp) -- (a1evaluate);
\draw[arrow] (a1evaluate) -- (a1update);

\draw[arrow,bend right=45] (a1update.east) to (a1model.east);


\node[process,right=6.0cm of a1model, anchor=north] (a2probs)
{Compute $\pi^k$ using $\{w_{k,n}, \mathbf{e}^{k,n}\}$ (Line \ref{line:probs_out_of_exp4})};

\node[process,below=of a2probs] (a2sample)
{Sample $I_k$ using $\pi^k$ (Line \ref{line:draw_ik})};

\node[process,below=of a2sample] (a2update)
{Update $\{w_{k+1,n}\}$ using $\{d^k_j\}$ (Line \ref{line:update_weights})};

\draw[arrow] (a2probs) -- (a2sample);
\draw[arrow] (a2sample) -- (a2update);

\draw[arrow,bend right=45] (a2update.east) to (a2probs.east);


\node[process,right=0.9cm of a1model,text width=2.3cm] (experts)
{Experts};


\node[plainbox,left=1.0cm of a1evaluate] (oracle)
{$F(\mathbf{x})$ oracle};


\draw[arrow] (a2sample) -- node[above,sloped]{$I_k,\pi^k_j$} (a1model);

\draw[arrow] (a1model) -- node[above,sloped]
{$d_j^k: j\in I_k$} (a2update);

\draw[bidir] (a1evaluate) to (oracle);

\draw[bidir] (a1model) to (oracle); 

\draw[bidir] (experts) to
node[below]{$\mathbf e^{k,n}$} (a2probs);


\node[group,fit=(a1model)(a1update),label=above:SAM method (\Cref{alg:sam})] (alg1box) {};

\node[group,fit=(a2probs)(a2update),label=above:Exp4 (\Cref{alg:exp4})] (alg2box) {};


\node[group,fit=(alg1box)(alg2box)(experts),inner sep=14pt] {};

\end{tikzpicture}
\end{figure}

Given this proposed combination of \Cref{alg:sam} and \Cref{alg:exp4}, we can now remark on the appropriateness of the result established in \Cref{cor:sublinear}.
There, we demonstrated that the cumulative regret of using \Cref{alg:exp4} to select $I_k$ over $K$ rounds with respect to a best expert only grows like $\sqrt{K}$; hence, the average (per-iteration) regret over $K$ rounds decays like $1/\sqrt{K}$.
In our past work \cite{SAMP2022}, we showed how \Cref{alg:sam} can be analyzed as a STORM method \cite{Chen2017}, and hence \Cref{alg:sam} inherits the same convergence rates.
In particular, given certain assumptions that amount to probabilistic bounds on the variance of the ameliorated models $\hat{m}^k$, we can show with probability one that $\|\nabla f(\xb_K)\|\leq \epsilon$ for $K\in\mathcal{O}(1/\epsilon^2)$ \cite{BCMS2018}.
This translates to a convergence rate on the order of $1/\sqrt{K}$, which matches the average regret rate implied by \Cref{cor:sublinear}.
Thus, the average regret from running \Cref{alg:exp4} does not dominate the worst-case convergence rate of \Cref{alg:sam}, and so the result of \Cref{cor:sublinear} appears appropriate for this setting.
This is admittedly a handwavy argument, but in future work, we will analyze these rates jointly, since the dynamics of the optimization method \Cref{alg:sam} and the sampling method \Cref{alg:exp4} are clearly intertwined. 

We now give three distinct suggestions for how to build experts.

\subsection{Online estimation of Lipschitz constants}\label{sec:lipschitz}
Here we give more detail to the idea of Lipschitz estimation discussed previously and employed in previous work\cite{SAMP2022}.
We will design an expert that maintains a running estimate $L_j$ of the (local) Lipschitz constant of each $\nabla F_j$. 
How the estimate $L_j$ is calculated depends on the model class used for constructing each model $m_j$. 
In the simple example of first-order Taylor models, upon every update of $\cb_j$ to $\xb^k$, we obtain a lower bound on the local Lipschitz constant $L_j$ via the secant 
$$\displaystyle\frac{\|\nabla F_j(\xb^k) - \nabla F_j(\cb_j)\|}{\|\xb^k-\cb_j\|}.$$
In the DFO setting, if we employ (perhaps underdetermined) quadratic interpolation models $m_j$, we can use the maximum absolute eigenvalue of the model Hessian $\nabla^2_{\xb} m_j(\xb;\xb^k)$ as an estimate of a local Lipschitz constant $L_j$. 
In either such setting, and to maintain some sense of locality, one can further consider using a mechanism like an exponential moving average of these coarse local Lipschitz constant approximations. 

Given the model class for constructing each $m_j$, and an associated means of (coarsely) approximating $L_j$, an expert can then use its estimate of $L_j$ and other contextual information generated by \Cref{alg:sam} to estimate an upper bound on the rewards in \cref{eq:max_discrepancy_in_delta_ball}. 
As a straightforward example, and as was done in \cite[Section 3.1]{SAMP2022}, when each $m_j$ is a first-order Taylor model, 
we can substitute the coarse approximation for $L_j$ to yield the approximate upper bound 
$$d^k_j \approx \displaystyle\frac{L_j}{2} \left(\Delta_k^2 + (\|\xb^k - \cb_j\| + \Delta_k)^2\right).$$
Similar upper bounds on rewards for other classes of models, including DFO models, can also be calculated, but are more complicated, see \cite[Section 3.3]{SAMP2022} for more details.
Common to the calculation of these bounds on $d^k_j$ is that they all depend only on contextual information generated by \Cref{alg:sam}, and they do not require additional evaluations of any $F_j$. 

An expert, indexed by $n$, that dynamically estimates local Lipschitz constants can then compute its advice $\mathbf{e}^{k,n}\in b\DD$ in the $k$th iteration of \Cref{alg:exp4} by plugging in its estimates of $d^k_j$ into the closed-form expression \cref{eq:opt_probs}.

\subsection{Offline surrogates or alternative low-fidelity models}\label{sec:surrogate}
Consider a setting where we have access to pre-trained surrogate models $\sigma_j\approx F_j$ for each $j\in\{1,\dots,p\}$, each trained on prior evaluations of $F_j$.
The notion of ``$\approx$" here is intentionally vague; each $\sigma_j$ only needs to be a reasonable approximation of $F_j$ in some (possibly non-quantifiable) sense held by a domain practitioner. 
Relatedly, if we have access to a lower-fidelity model $\sigma_j\approx F_j$ for each $j\in\{1,\dots,p\}$ (where, again, ``$\approx$" is intentionally vague), then we can apply the same ideas we propose here to yield an expert, but we note that no explicit correlation between $\sigma_j$ and $F_j$ must be identified or modeled. 
We minimally assume that queries to an oracle for $\sigma_j$ are relatively negiglible in cost compared to queries to $F_j$. 

Within this setting, an expert with knowledge of the models $\{\sigma_j:j=1,\dots,p\}$ can estimate the effect of updating the $j$th model by substituting any required information concerning $F_j$ with information concerning $\sigma_j$, thus giving the expert an estimate of each $d^j_k$. 
For example, in the case where $m_j$ is a first-order Taylor model, the expert can directly compute the reward $d^k_j$ in \cref{eq:max_discrepancy_in_delta_ball} by solving 
\begin{align*}
 d^k_j = &   \displaystyle\max_{\sv:\|\sv\|\leq\Delta_k} \left|\left(\sigma_j(\xb^k) + \nabla \sigma_j(\xb^k)^\top\sv\right) - m_j(\xb^k+\sv; \cb^{k-1}_j)\right|\\
= & \displaystyle\max_{\sv:\|\sv\|\leq\Delta_k} \left|
\left(\sigma_j(\xb^k) - \sigma_j(\cb^{k-1}_j)\right) + (\nabla \sigma_j(\xb^k) - \nabla F_j(\cb^{k-1}_j)^\top \sv\right|\\
= & |\sigma_j(\xb^k) - \sigma_j(\cb^{k-1}_j)| + \Delta_k\|\nabla \sigma_j(\xb^k) - \nabla F_j(\cb^{k-1}_j)\|.
\end{align*}
Similarly, in the DFO model case, the expert can approximate the reward $d^k_j$ in \cref{eq:max_discrepancy_in_delta_ball} by 1) identifying all of the geometry points -- call this set $\Yb$ -- that would have to be evaluated if the model center $\cb^{k-1}_j$ was updated to $\xb^k$, 2) evaluating $\sigma_j(\yb)$ at each $\yb\in\Yb$, and then 3) tentatively updating the model with these $(\yb, \sigma_j(\yb))$ data pairs to yield $m^{\sigma}_j(\xb;\xb^k) \approx m_j(\xb;\xb^k)$.
Then, the expert can estimate $d^k_j$ by (approximately) solving
\begin{equation}\label{eq:max_discrepancy_dfo}
d^k_j\approx \displaystyle\max_{\sv:\|\sv\|\leq\Delta_k} \left|m^{\sigma}_j(\xb + \sv;\xb^k) - m_j(\xb^k+\sv;\cb^{k-1}_j) \right|.
\end{equation}
We recall again that the tractability of computing the right hand side of \cref{eq:max_discrepancy_dfo} depends on the choice of model class $m_j$.
The right hand side of \cref{eq:max_discrepancy_dfo} admits a closed-form solution, for example, when $m_j$ are linear interpolation models, but obtaining the global maximum of the right hand side of \cref{eq:max_discrepancy_dfo} becomes generally more difficult as the model complexity of $m_j$ increases. 
As with the Lipschitz constant-estimating expert, the expert in the offline surrogate setting can once again produce advice $\mathbf{e}^{k,n}\in b\DD$ by plugging the estimated rewards $d^k_j$ into \cref{eq:opt_probs}. 

\subsection{Human and AI experts}
One could imagine that a human expert (in particular, a domain scientist with deep knowledge of what the oracle $F(\xb)$ represents) could examine all the data generated by \Cref{alg:sam} up until the $k$th iteration. 
Unlike the experts employing Lipshitz estimation or surrogate models defined in \Cref{sec:lipschitz} and \Cref{sec:surrogate} respectively, the human expert likely has their own beliefs about how well they can predict values of $F_j(\xb^k)$ for different $j\in\{1,\dots,p\}$ based on the data generated by \Cref{alg:sam} \emph{in addition to} the context of what each parameter in $\xb$ and what each component function $F_j$ represent.
A human expert's ability to translate these beliefs into a set of probabilities $\mathbf{e}^{k,n}\in b\DD^p$ is likely heuristic, but not without practical value. 

In the same sense, a foundation model -- ideally, either fine-tuned with some curated training dataset providing meaningful context about what the parameters $\xb\in\Reals^{dim}$, and about what the observables $\{F_j: j=1,\dots,p\}$, represent, or provided a suitable means for retrieval-augmented generation (RAG) -- can provide a probabilistic means of mapping the data generated by \Cref{alg:sam} directly into probability vectors. 
This could be achieved, for instance, by simply providing a foundation model with some appropriately formatted history of a current run of \Cref{alg:sam}, along with context of what $\xb, F(\xb)$ represent. 
The foundation model could then be prompted to provide a set of relative importance weights for updating models, which might potentially have to be normalized into $b\DD^p$ in a post-processing step. 

We note that keeping a human in the loop as an expert\footnote{Noting, of course, that a human expert indexed by $n$ might be less than thrilled to discover that they are being assigned a relatively low weight $w_{k,n}$ in later iterations of \Cref{alg:exp4}.}, or employing a foundation model as an expert, may appear less mathematically rigorous than the other two methods we have proposed here, but this underscores the generality of the contextual bandit framework -- \emph{any procedure, which optionally ingests context, to produce vectors in $b\DD^p$ is a valid expert.} 

\section{Implementation and Numerical Experiments}\label{sec:experiments}
We begin this section by describing some key features of our implementation of \Cref{alg:sam}, which makes use of \Cref{alg:exp4} as described in \Cref{sec:experts}. 
The implementation used in our experiments is available at the static repository \url{https://github.com/mmenickelly/bandit_sampounders}.
We will then conclude this section by providing some numerical experiments, which can be replicated using the code in the same repository. 

\subsection{Implementation}\label{sec:implementation}
Due to our primary research motivations, we have only experimented with a variant of \Cref{alg:sam} for derivative-free optimization.
Towards this end, we built a SAM method, which we name \texttt{SAM-POUNDERS}, on top of the existing code \texttt{POUNDERS} (Practical Optimization Using No Derivatives Exploiting Recognizable Structure) available in the \texttt{IBCDFO} (Interpolation-Based Composite Derivative-Free Optimization) repository \cite{ibcdfo}. 
\texttt{POUNDERS} is a model-based derivative-free optimization method designed to solve problems of the form 
\begin{equation}\label{eq:pounders_prob}
\displaystyle\min_{\xb\in\Reals^{dim}} h(F(\xb)) \quad \text{s. to} \quad \ellb\leq\xb\leq \ub,
\end{equation}
where
$F:\Reals^{dim}\to\Reals^p$ is a black-box function, 
$h:\Reals^p\to\Reals$ is a twice-differentiable glass-box function, and
$\ellb,\ub\in\Reals^{dim}$ define box constraints.
\texttt{POUNDERS} was originally designed for black-box nonlinear least squares minimization where $h=\|\cdot\|^2$; 
we note that $h$ can also be a function that sums the entries of its $p$-dimensional argument, hence making \texttt{POUNDERS} appropriate for solving \cref{eq:sum}. 

Being a model-based trust-region method, \texttt{POUNDERS} will maintain models $\{m_j(\xb): j=1,\dots,p\}$ via interpolation.
By default, \texttt{POUNDERS} will employ (underdetermined) quadratic interpolation models so that 
$$m_j(\xb) = h(F_j(\cb_j)) + \gb_j^\top (\xb-\cb_j) + \frac{1}{2}(\xb-\cb_j)^\top \Hb_j(\xb-\cb_j)$$
for $j=1,\dots,p$ where each $\gb_j, \Hb_j$ are fit via some interpolation scheme. 
When a model $m_j$ is updated in \Cref{alg:sam}, a \texttt{POUNDERS} subroutine is called, which constructs a new model centered at $\xb^k$, re-utilizing as many past function evaluations as possible.
The \texttt{POUNDERS} subroutine is designed to be judicious in evaluating $F$, so as to limit the number of function evaluations performed when updating a model. 
The \texttt{POUNDERS} model-updating subroutines employ a rank-revealing QR decomposition of displacements of previously-evaluated points from $\xb^k$ to determine a subset of the previously-evaluated points to use as interpolation points; a previously-evaluated point is replaced with a new point only when its corresponding $R$-factor pivot is deemed too small in absolute value. 
This process is described in detail in \cite[Section 4]{SW08}.
We note that this implies an update to $m_j$ can involve anywhere between $1$ and $dim+1$ many function evaluations of $F_j$. 

In our experiments, we will exclusively let $h=\|\cdot\|^2$.
As a result, we update our definitions of the average and ameliorated models as
\begin{equation}\label{eq:pounders_ameliorated}
\bar{m}^k(\xb) := \displaystyle\sum_{j=1}^p m_j(\xb;\cb^k_j)^2, 
\quad \hat{m}^k(\xb) := \bar{m}^{k-1}(\xb) + \displaystyle\sum_{j\in I_k} \frac{m_j(\xb;\xb^k)^2 - m_j(\xb;\cb_j^{k-1})^2}{\pi^k_j}
\end{equation}
respectively,
and observe that despite these new definitions, we maintain the crucial property that
\begin{equation}\label{eq:expected_is_ideal}
    \mathbb{E}_{I_k}\left[\hat{m}^k(\xb)\right] = \displaystyle\sum_{j=1}^p m_j(\xb;\xb^k)^2.
\end{equation}
Thus, in the spirit of \texttt{POUNDERS} being able to exploit known derivative information of $h$, 
we note that in the trust-region subproblem appearing in \lineofalg{line:trsp}{alg:sam}, we can use chain rule to derive
$$\nabla_{\xb}\left[m_j(\xb;\cb_j)^2\right] = 2m_j(\xb;\cb_j)\left(\gb_j + \Hb_j(\xb-\cb_j)\right),$$
for each $j=1,\dots,p$. 
We then take a $\pi^k$-weighted sum of such gradient approximations according to the ameliorated model definition in \cref{eq:pounders_ameliorated} to define the model gradient $\nabla_{\xb}\hat{m}^k(\xb)$ for use in the trust-region subproblem solver. 

We want to stress that such an approach will \emph{not} work for general $h$ functions.
It is because $h=\|\cdot\|^2$ maintains a finite-sum structure that the unbiasedness property for amelioration, \cref{eq:expected_is_ideal}, is maintained.
Identifying how to make a SAM method work with more general (differentiable) $h$ functions is a topic for future work. 

\paragraph{Randomized estimates of $F(\xb^k+\sv^k)$.} As remarked in our initial discussion of \Cref{alg:sam}, we chose to present \Cref{line:rho_compute} as if it involved a deterministic computation of $f(\xb^k+\sv^k)$. 
However, we in fact also randomize this computation by generating a second sample $J_k\subseteq\{1,\dots,p\}$ according to a second set of probabilities $\pi^{k, J}$, computing any previously uncomputed values of $\{(F_j(\xb^k), F_j(\xb^k+\sv^k)): j\in J_k\}$,
and estimating in the $\rho_k$ test of \Cref{line:rho_compute}
\begin{equation*}
\begin{array}{c}
f(\xb^k) \approx \bar{m}^{k-(1/2)}(\xb^k) 
+ \displaystyle\sum_{j\in J_k} \frac{F_j(\xb^k) - m_j(\xb^k; \cb_j^{k-(1/2)})}{\pi^{k,J}_j},\\
f(\xb^k + \sv^k) \approx \bar{m}^{k-(1/2)}(\xb^k + \sv^k) 
+ \displaystyle\sum_{j\in J_k} \frac{F_j(\xb^k + \sv^k) - m_j(\xb^k + \sv^k; \cb_j^{k-(1/2)})}{\pi^{k,J}_j},
\end{array}
\end{equation*}
where we use the $k-(1/2)$ superscript to index the model centers $\cb_j$ after having already performed the center updates in 
\lineofalg{line:model_updates}{alg:sam}.
We also use the $k-(1/2)$ superscript in the average model $\bar{m}$ implied by these same model centers.  

Our methods for computing $\pi^{k,J}$ are analogous to our methods for computing $\pi^{k}$ for each expert. 
As an example, we look at the case of using offline surrogates presented in \Cref{sec:surrogate}. 
In \cref{eq:max_discrepancy_dfo}, instead of maximizing the estimated model discrepancy over the current trust-region $\{\sv: \|\sv\|\leq\Delta_k\}$, we would instead only maximize over the two-point set $\{\xb^k, \xb^k+\sv^k\}$ to compute values of $d^k_j$. 
An analogous consideration is made for the estimation of model discrepancy via Lipschitz estimation presented in \Cref{sec:lipschitz}.

We note that conditioning the sample $J_k$ on the sample $I_k$ may not be theoretically justified.
We assumed in past work \cite{SAMP2022} that the ameliorated models for estimating $\rho_k$ were independent of the ameliorated models $\hat{m}_k$ employed in the trust-region subproblem. However, we have seen that this dependence of $J_k$ on $I_k$ works well in practice. Demonstrating the theoretical conditions under which this conditioning can be shown not to prevent convergence is an avenue for future work. 

\paragraph{Reward scaling.} As seen in \Cref{cor:sublinear}, cumulative regret scales linearly with $D$, and our choice of exploration parameter $\gamma$ ought to scale with $1/D$, where $D$ is an upper bound on $d^k_j$ across all $K$ rounds (iterations) of \Cref{alg:exp4} (\Cref{alg:sam}), and for all $j=1,\dots,p$. 
Being able to scale reward magnitudes, especially in early iterations, is also very much a practical concern since disproportionately large values of $d^k_j$ can create severe imbalances in the weight updates in \Cref{alg:exp4}. 

We found that maintaining, by an exponential moving average, an online estimate $D_k$ of the largest observed reward up until the $k$th iteration was a reasonable practical mediation.
In particular, in the first iteration we store $D_1:=\max\{d^1_j: j\in I_1\}$. 
In every subsequent iteration, we update
$D_k := \delta D_{k-1} + (1-\delta)\max\{d^{k-1}_j: j\in I_{k-1}\}$,
where $\delta\in(0,1)$ is an algorithmic constant; we set $\delta=0.8$ in all our experiments. 
Then, whenever a batch of rewards is received in \lineofalg{line:ask_for_reward}{alg:exp4}, all rewards are first multiplied by $D_k^{-1}$, with the intended effect being that the largest rewards $\{d^k_j: j\in I_k\}$ are $\mathcal{O}(1)$.

We maintain a separate estimate $D_{k,J}$ of the maximum rewards pooled from estimating both $F_j(\xb^k)$ and $F_j(\xb^k+\sv^k)$, but otherwise follow the exact same procedure. 

\paragraph{Enforcing $|I_k|=b$.} For the sake of presentation, we operated until this point under the assumption that we only enforce $|I_k|=b$ in expectation over all draws of $I_k$ governed by $\pi^k$. 
In practice, and in order to ensure that we do not under-utilize or (worse) over-utilize parallel resources, we enforce $|I_k|=b$ explicitly. 
This is accomplished via conditional Poisson sampling, see \cite[Algorithm 2]{SAMP2022} for details. 
In effect, we derive a secondary distribution from $\pi^k$ and perform rejection sampling with this distribution until we draw $I_k$ of size $b$.
The derivation of the secondary distribution ensures that the key property $\mathbb{P}\left[j\in I_k\right] = \pi^k_j$ is maintained for all $j=1,\dots,p$. 

\paragraph{Algorithmic parameters.}
All parameters governing trust-region dynamics in \Cref{alg:sam} are inherited from the default values in \texttt{POUNDERS} and are considered fairly standard across literature for trust-region methods. 
Moreover, all parameters governing model building and updating in \Cref{alg:sam} are inherited from the default values in \texttt{POUNDERS}. 

For parameters concerning \Cref{alg:exp4}, we previously discussed how we perform rewards scaling. 
In all the experiments we present here, we were particularly aggressive and set the batch size $b=1$ for all tests unless otherwise stated.  
We note, however, that $b$ should be determined in general based on available parallel computational resources. 
To set $\gamma$, we found that a tuning derived from \Cref{cor:sublinear} worked well in practice. 
In particular, we estimate $K$ as our total budget for component function evaluations in a single run of an optimization method divided by the batch size $b$. 
This leads to our choice
$\gamma = \sqrt{\left(p\ln(\max\{2,N\})\right)/\left(b(budget)\right)}$,
where $N$ is the number of experts.

\subsection{Numerical Experiments}

We designed a set of three experiments to test the use of each of the three types of experts suggested in \Cref{sec:experts} within \texttt{SAM-POUNDERS}, and we follow the same organization in this section. 

\subsubsection{Lipschitz estimation}\label{sec:lipschitz_exp}
When we employ Lipschitz estimation as the only expert suggesting sampling distributions in \texttt{SAM-POUNDERS}, we are effectively testing the same method that was used in our prior work \cite{SAMP2022}. 
Given the regret bounds associated with \Cref{alg:exp4}, 
we are especially interested in discovering empirical evidence demonstrating how this past mode of generating sampling distributions compares to both 1) using a uniform sampling distribution (that is, the uniform expert of \Cref{thm:sublinear_regret}), and of course 2) using \Cref{alg:exp4} to mix the two experts to form distributions. 

Towards this end, we employ the Mor{\'e}-Wild (MW) benchmarking test set, introduced in \cite{JJMSMW09} and available through the \texttt{BenDFO} repository \cite{BENDFOCode}.
The MW test set consists of 53 vector-valued functions of the form $F:\Reals^{dim}\to\Reals^p$. 
In the problem description given in \cref{eq:pounders_prob}, we let $h=\|\cdot\|^2$.
Due to the design of the MW test set, we know that this composition of $h\circ F$ admits a set of well-defined test problems. 
In the MW test set, $2\leq dim\leq 12$ and $2\leq p \leq 45$. 
For 10 random seeds for each of the 53 problems, we ran \texttt{SAM-POUNDERS} with 1) a single expert defined by Lipschitz estimation, 2) a single expert defined by uniform sampling, and 3) both experts' advice being mixed by \Cref{alg:exp4}.
As a deterministic comparator, we ran \texttt{POUNDERS}. 
Based on our motivations, and adhering to an assumption that all $\{F_j: j=1,\dots,p\}$ are equally expensive to evaluate, we count the total number of single evaluations of \emph{any} $F_j$ as our budget unit. 
Each method was given a budget of $50dp$ evaluations of any $F_j$.
As a reference, \texttt{POUNDERS}, which in effect deterministically selects $I_k=\{1,\dots,p\}$ on every iteration, was therefore effectively given a budget of $50d$ evaluations of $F(\xb)$. 

We compare the four solvers via performance profiles \cite{JJMSMW09}.
In a performance profile, we indicate a problem $prob$ as being solved by a given solver in iteration $k$ when 
$$f(\xb^k) \leq f^{prob}_{best} + \tau (f(\xb^0) - f^{prob}_{best}),$$
where $\tau < 1$ is the \emph{solvedness tolerance}, $\xb^0$ is a problem-specific starting point shared by all solvers applied to $prob$, and $f^{prob}_{best}$ is the best function value found by any solver applied to $prob$. 
We note that because \texttt{POUNDERS} solves all 53 MW problems, this definition of $f^{prob}_{best}$, and hence this definition of solvedness, is reasonable. 
A performance profile plots one line per solver, each of which should be interpreted as an empirical cumulative density function, the support of which is $[1,\infty)$ and represents the ratio of the budget spent by the solver to achieve solvedness on a given problem to the least budget spent by any solver to achieve solvedness on the same problem. 
That is, for each solver, scanning across the $x$-axis from left to right, the $y$-axis value is incremented by the reciprocal of the number of problems in the test set each time a realized ratio between budgets is encountered on the $x$-axis. 

\begin{figure}
    \caption{\label{fig:lipschitz} Performance profiles from experiments comparing variants of \texttt{SAM-POUNDERS} that use Lipschitz estimation as an expert. Left image corresponds to solvedness tolerance $\tau=10^{-3}$, right image corresponds to $\tau=10^{-7}$.}
    \centering
    \includegraphics[width=.49\linewidth]{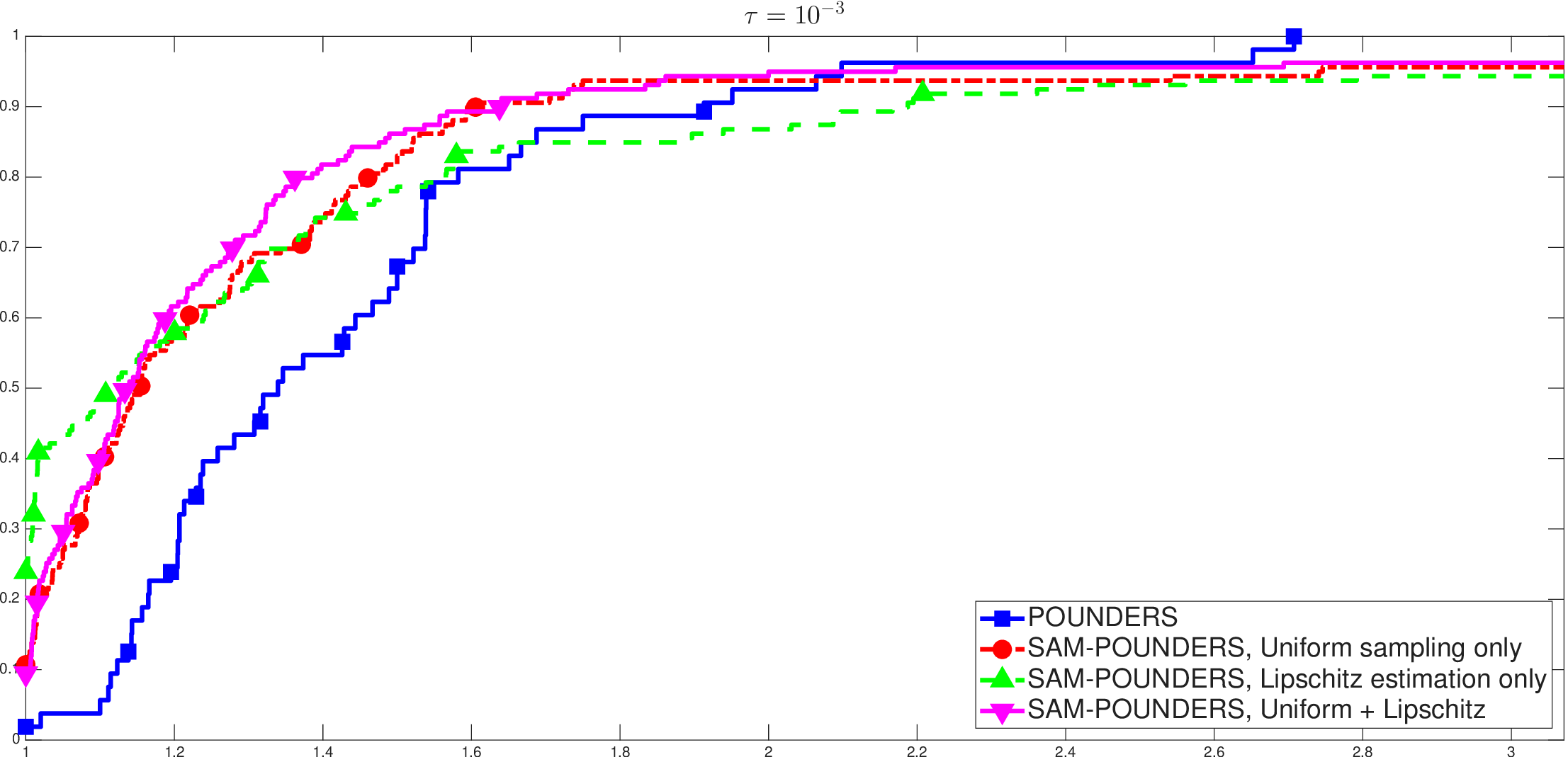}
    \includegraphics[width=.49\linewidth]{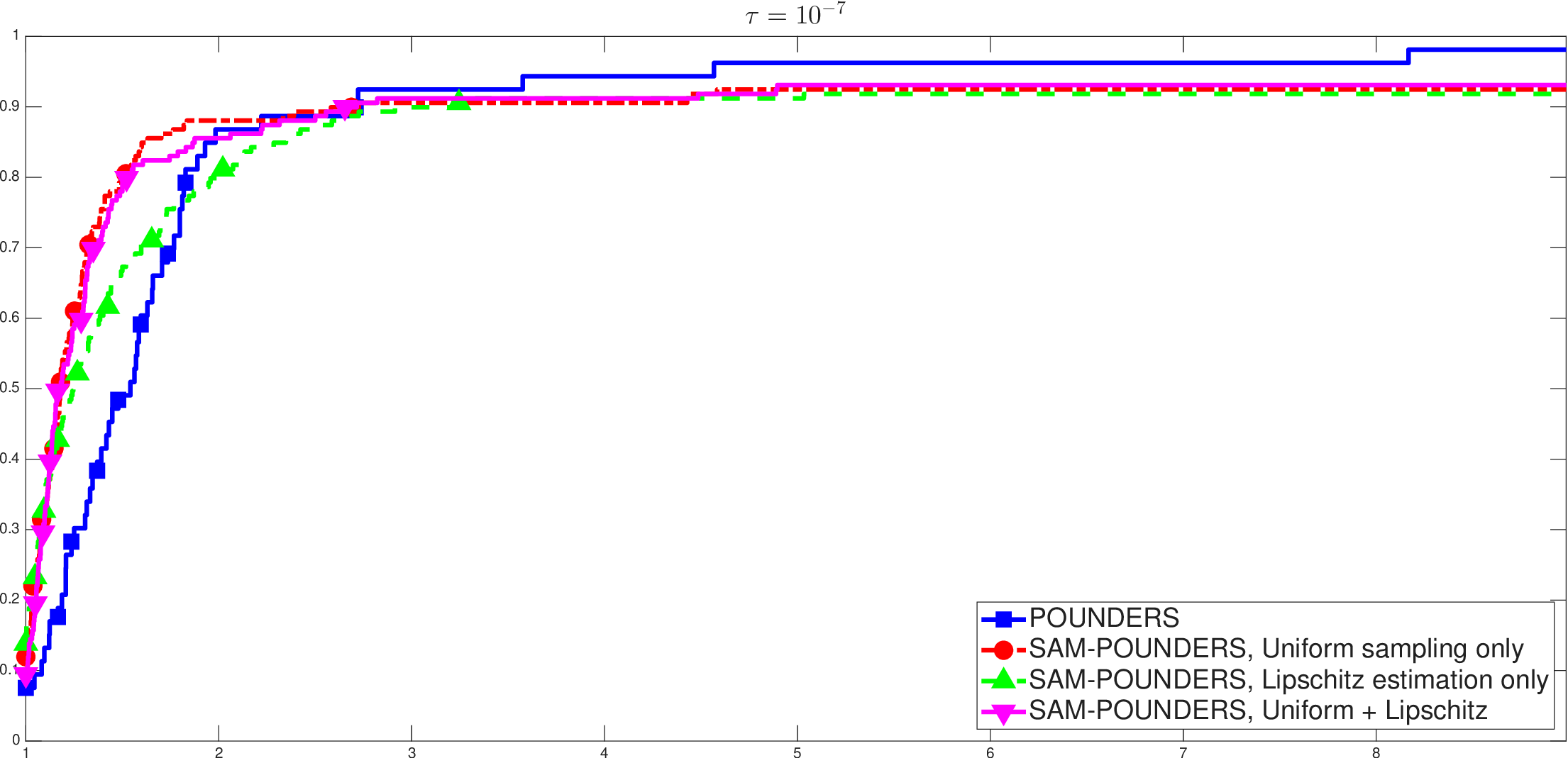}
\end{figure}
Performance profiles for this experiment are displayed in \Cref{fig:lipschitz}. 
Here, we see that at a lower solvedness tolerance ($\tau=10^{-3}$), employing Lipschitz estimation as the only expert solves the greatest fraction of problems fastest. 
However, its robustness relative to uniform sampling quickly drops out. 
In this setting, we see that mixing uniform sampling and Lipschitz estimation yields a very competitive method, serving as the fastest solver of the four tested, and solving over $90\%$ of the problems within budget -- only deterministic \texttt{POUNDERS} solves $100\%$ of the problems here. 
At a tighter solvedness tolerance ($\tau=10^{-7}$), we see that uniform sampling is generally preferable to Lipschitz estimation, and mixing the two experts results in little regret relative to either. 

\subsubsection{Offline surrogates}
In this set of experiments, we will compare \texttt{POUNDERS} and \texttt{SAM-POUNDERS} employing only a uniform expert as in \Cref{sec:lipschitz_exp} to \texttt{SAM-POUNDERS} employing an expert with knowledge of a pretrained surrogate, as discussed in \Cref{sec:surrogate}, as well as to \texttt{SAM-POUNDERS} employing both the uniform expert and the surrogate-defined expert. 
We reuse the MW test set from \Cref{sec:lipschitz_exp}.
To simulate having a pretrained offline surrogate, for each pair of random seed generated and problem in the test set, we generate 500 samples from $[\ellb,\ub]\in\Reals^{dim}$ according to a Latin hypercube design. 
The problems in the MW test set provided in \texttt{BenDFO} do not impose bound constraints $[\ellb,\ub]$ by default; we employ the same procedure discussed in \cite[Section 4.1]{larson2024structure} to generate them by observing the trajectories $\{\xb^k\}$ generated by the deterministic run of \texttt{POUNDERS}. 
Then, we evaluate $F$ at these 500 samples.
For each $j=1,\dots,p$, we construct a surrogate $\sigma_j$ of $F_j$ by employing an off-the-shelf radial basis function interpolation, in particular, one provided by \texttt{Matlab}. 
The collection of surrogates $\{\sigma_j: j=1,\dots,p\}$ are used by an expert as described in \Cref{sec:surrogate}. 

\begin{figure}
    \caption{\label{fig:surrogate} Performance profiles from experiments comparing variants of SAM-POUNDERS that use an expert employing pretrained surrogates. Left image corresponds to solvedness tolerance $\tau=10^{-3}$, right image corresponds to $\tau=10^{-7}$.}
    \centering
    \includegraphics[width=.49\linewidth]{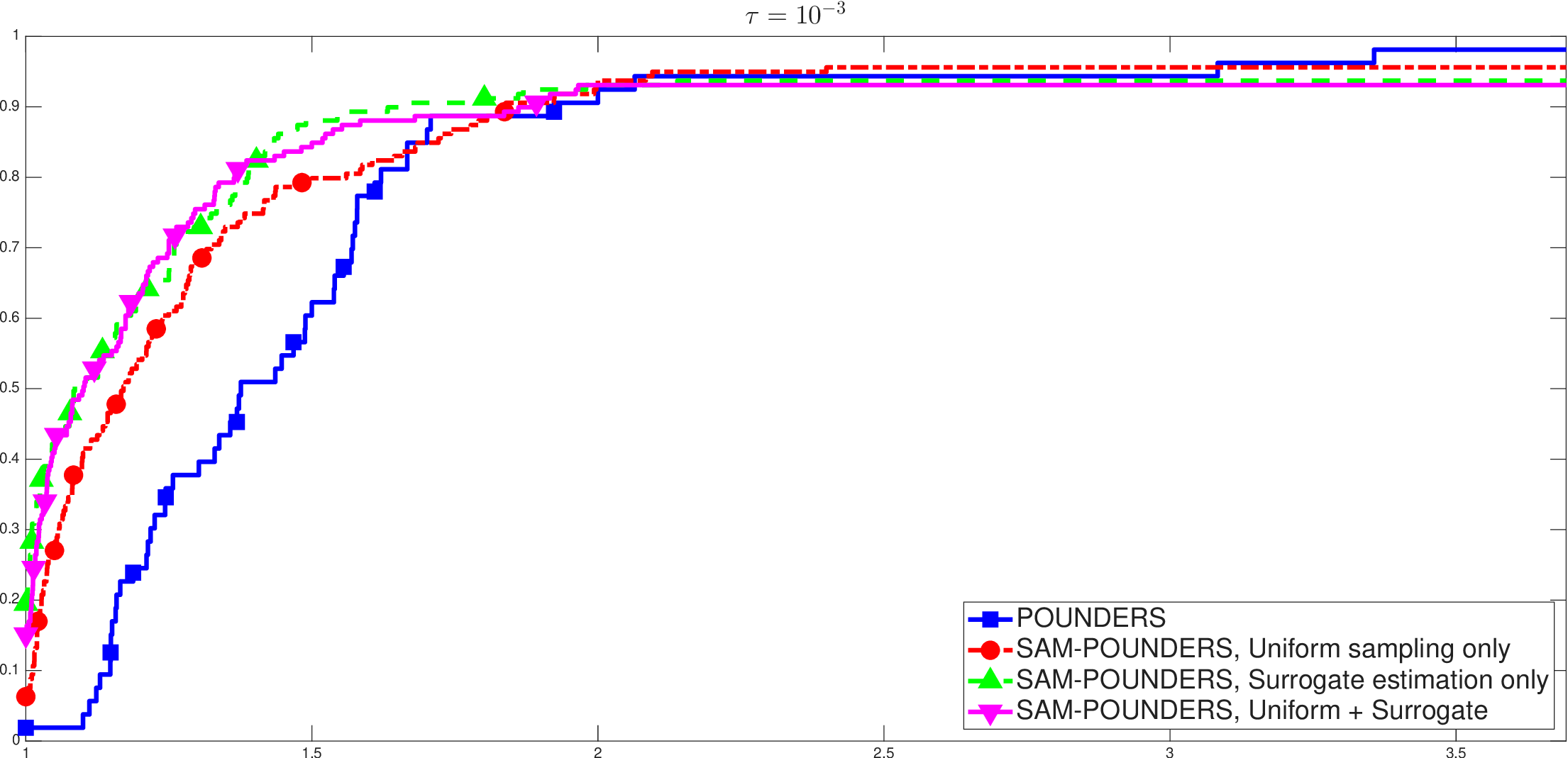}
    \includegraphics[width=.49\linewidth]{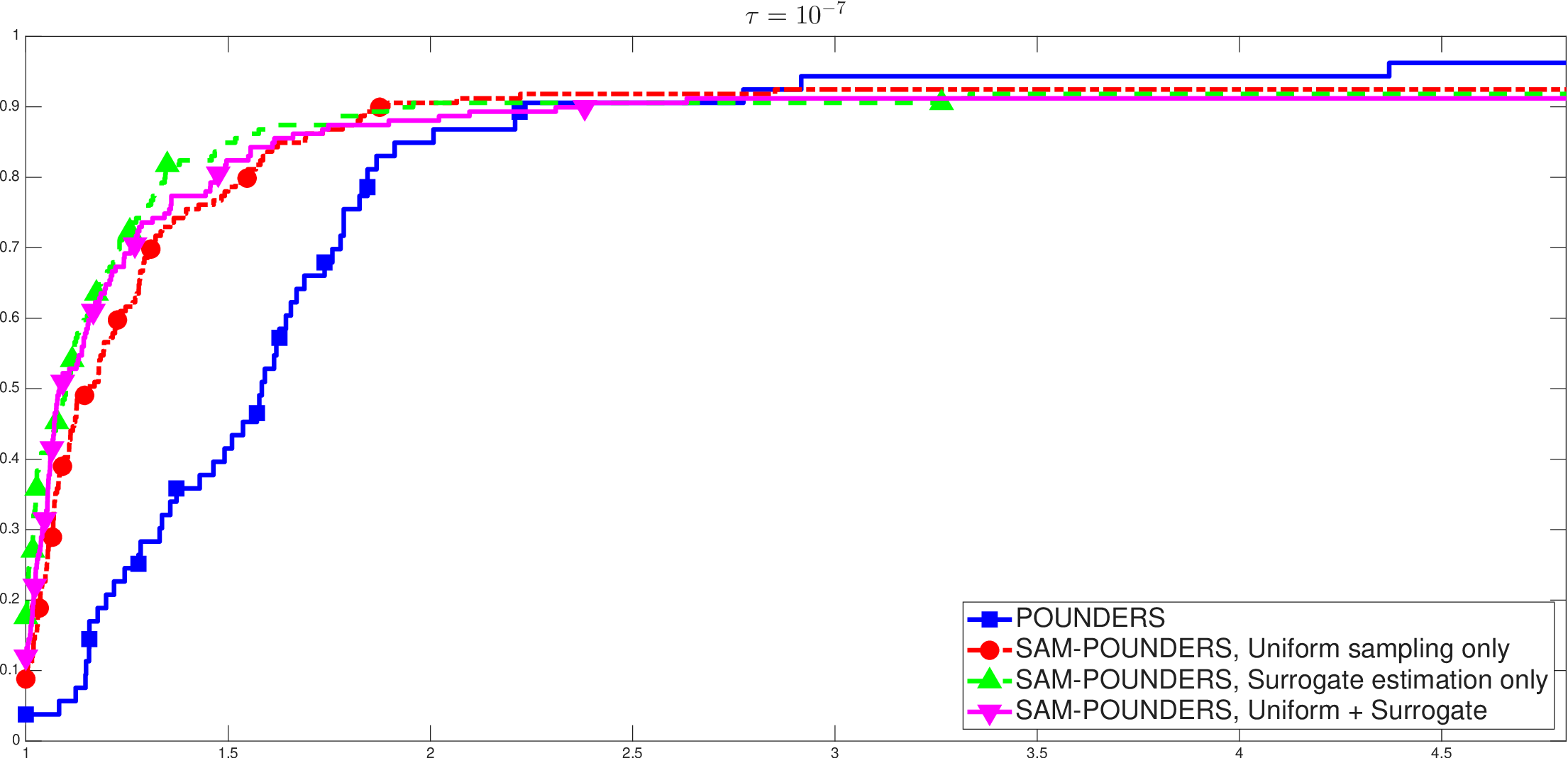}
\end{figure}
The results of this experiment are displayed in \Cref{fig:surrogate}
At both levels of the solvedness tolerance displayed, we see roughly the same trend; the surrogate-informed expert is generally a more useful expert than uniform sampling. 
Importantly, mixing the two experts does not result in too much regret over committing to using only the surrogate-informed expert. 

\subsubsection{AI Expert}\label{sec:gemini}
We experiment here with the use of an expert derived from a commercial LLM; in this experiment, we employed Gemini 3.1 Pro. 
The intention here is not to claim that any particular commercial LLM is best-suited for deriving an expert within our framework.
In fact, by the nature of using a cloud-based LLM, this experiment is impossible to reproduce exactly, and as a result, is not completely rigorous. 
The intention of this experiment is to provide anecdotal evidence concerning the use of general-purpose (not fine-tuned with any particular data set) foundation models to generate expert advice.

For this experiment, we considered a synthetic problem from elementary Newtonian physics; this was motivated by the expectation that a general-purpose foundation model would likely have had relevant information about such a problem contained in its training set. 
In particular, we created a data-fitting problem in which we provided the forward model for the displacement of a 1D damped harmonic oscillator, parameterized by four phenomenological constants representing amplitude ($A$), a damping coefficient ($\gamma$), angular frequency ($\omega$) and initial phase offset ($\phi)$.
In closed form, this forward model is well-known to be 
$$y(t;A,\gamma,\omega,\phi) = Ae^{-\gamma t}\cos(\omega t + \phi) \quad \forall t\geq 0.$$
We fixed $p$ many equally-spaced time points $\{t_j: j=1,\dots,p\}$ and a ``true" set of parameters $\xb^{true} = (A^{true}, \gamma^{true}, \omega^{true}, \phi^{true})$ to generate a corresponding set of synthetic measurements $\{y^{obs}_j := y(t_j;\xb^{true}) + \xi_j: j=1,\dots,p\}$, where $\xi_j$ is a zero-mean Gaussian noise term with variance $A_j/100$. 
Recalling the notation from \cref{eq:pounders_prob}, the resulting nonlinear least squares minimization problem consists of setting $h=\|\cdot\|^2$ and $F$ defined entrywise by $F_j(\xb) = y(t_j;\xb) - y^{obs}_j$. 

In every iteration of \Cref{alg:sam}, Gemini 3.1 Pro was given the verbatim prompt displayed in \Cref{fig:prompt}. 
\begin{figure}[t!]\caption{\label{fig:prompt} Prompt supplied to Gemini Pro 3.1 in every iteration of \texttt{SAM-POUNDERS}. Some sentences are highlighted here (but not in the actual prompt used in the experiment) to be discussed in inline text.}
\begin{framed}
    I am fitting four phenomenological parameters (amplitude, damping coefficient, angular frequency, and initial phase offset) of a damped harmonic oscillator to experimental data obtained by measuring the displacement at \texttt{($p$)} equally spaced candidate time points from time \texttt{($t_0$)} seconds to \texttt{($t_1$)} seconds. I am running a forward model code that takes as input both 1) a subset of these time points and 2) a set of phenomenological parameters. The forward model returns simulated measurements of displacement at the input time points given the input parameters (there is simulated measurement error). \textbf{Let's pretend that evaluating the forward model is a very expensive operation (I know it isn't, but humor me).} In an outer loop optimization method, I am maintaining \texttt{($p$)} separate local quadratic models of the displacement as a function of the four parameters, one model per each time point. Each local quadratic model is characterized by a center point, around which the model is locally valid. In each iteration of the optimization method, I can update a small batch of the quadratic models so that their centers are all updated to a given, incumbent set of parameters. At the end of this prompt, I will provide a json that shows, for each value of time t, at what parameter values the quadratic model of the corresponding model is currently centered. With this setup in mind, give me an array of length \texttt{($p$)} of nonnegative real-valued numbers that sum to 1.0; each entry of the array should correspond to a probability with which I should update the model center to the incumbent parameter estimate values \texttt{(json containing $x^k$)} at the corresponding time point. \textbf{I will fix your work if I have to, but please try to make sure all entries of the array are nonnegative and sum to 1.0. Also, please double check that there are exactly \texttt{($p$)} entries in the array. Think hard about the nonlinear relationship between time t and the four phenomenological parameters. Also think hard about the fact I am using quadratic models and what the second-order Taylor expansion might tell you about the quadratic model's accuracy at the incumbent.} Please respond only with a json with two fields, one containing the array, labelled 'array', and one containing your reasoning, labelled 'reasoning'. In your reasoning, if you used any formulas to derive your array, then please tell me how you interpreted the parameter names in the json that I provided you as symbols in your formulas. Please keep in mind that I will be giving you this same prompt multiple times, and so you will be able to see which models I decided to update. The json with the model center data I promised is: \texttt{(json)}. 
    \end{framed}
    \end{figure}
Some remarks on this prompt are in order here. 
Even with the prompt's reminders to ensure that the returned array was indeed in $\DD^p$, this was frequently violated; some entries would be negative (but relatively small in magnitude), and sometimes there would be more than $p$ entries in the array (but notably, never fewer than $p$ entries). 
When post-processing the returned array, we computed the absolute value of the array. 
Additionally, whenever an array contained more than $p$ entries, we truncated the array to length $p$. 

As a positive, we highlight our inclusion of a sentence ``Let's pretend that evaluating the forward model is a very expensive operation (I know it isn't, but humor me).". 
This was included because Gemini Pro 3.1 correctly identified that the forward model was easy to compute, and that our proposed approach of fitting quadratic models was unnecessary. 
However, even without this sentence, Gemini Pro 3.1 still returned an array as instructed, but would almost always mention this in its reasoning. 
This type of behavior has important implications for more useful science use cases, since it suggests that the foundation model might suggest better classes of models (in this trivial case, the exact forward model) for data-fitting tasks, even when not prompted to do so. 

We also highlight the necessity of the sentence ``Think hard about the nonlinear relationship between time $t$ and the four phenomenological parameters" in performing this experiment.
In the absence of this sentence, Gemini Pro 3.1 (correctly) guessed that the ``outer-loop optimization method" in question was a nonlinear least squares algorithm. 
In turn, Gemini Pro 3.1 exploited information in the provided json to approximate the current residuals and tried to propose an importance sampling scheme that sampled proportionally to observed residual magnitudes.
Such a sampling strategy is a coarse approximation to the Lipschitz estimation discussed in \Cref{sec:lipschitz_exp} and effectively treats the forward model $F(\xb)$ as a black box, which was not the desired effect. 
Additionally, we had to include the sentence ``Also think hard about the fact I am using quadratic models and what the second-order Taylor expansion might tell you about the quadratic model's accuracy at the incumbent." Without this sentence, Gemini Pro 3.1 effectively ignored the presence of quadratic models altogether and only took into consideration the Hessian of the forward model, completely ignoring the distance between the incumbent $\xb^k$ and model centers $\{\cb^k_j\}$.  
However, once provided with these reminder sentences to consider nonlinearities of the forward model and how they might be modeled by quadratic functions, example reasoning text supplied by Gemini Pro 3.1 looked like that displayed in \Cref{fig:reasoning}, which we have relegated to the appendix for reference.

For the experiment, we set the ground truth parameters $\xb^{true} = (A^{true}, \gamma^{true}, \omega^{true},\phi^{true}) = \\(1.5, 0.2, 2.0, \pi/2)$. We set $p=40$, and let the batch size for \texttt{SAM-POUNDERS} variants be $b=4$; this latter decision was based in part on the fact that there are only four parameters in the problem, and so every batch potentially admits a determined system of equations.  
We ran \texttt{POUNDERS} as a baseline and compared against three variants of \texttt{SAM-POUNDERS}, each with only one expert: one employing the Lipschitz estimation scheme of \Cref{sec:lipschitz_exp}, the uniform expert, and the expert derived by prompting Gemini Pro 3.1 as in \Cref{fig:prompt}. 
In order to guarantee that all solvers identified the same local minimum, we chose as a common initial point $\xb^0 = \xb^{true} + 0.1 \eb,$ where $\eb\in\Reals^4$ denotes the vector of all ones. 
Results are displayed in \Cref{fig:expert_exp}. 
We see a practical advantage, in terms of number of component function evaluations performed, of using \texttt{SAM-POUNDERS} with Gemini Pro 3.1 sampling to achieve four or five digits of accuracy, at which point uniform random sampling is equally as good. 
For seven or more digits of accuracy, \texttt{POUNDERS} is the preferred solver, although we remind the reader that with observations exhibiting Gaussian noise with variance on the order of $10^{-2}$, seven digits of accuracy are probably not meaningful. 
\begin{figure}[t!]
    \caption{\label{fig:expert_exp} Comparing performance of three variants of \texttt{SAM-POUNDERS} against \texttt{POUNDERS} on the problem of recovering parameters for a 1D damped harmonic oscillator. Center lines denote median performance over 30 trials, outer bands denote $10-90\%$ percentile performance.}
    \centering
    \includegraphics[width=.99\textwidth]{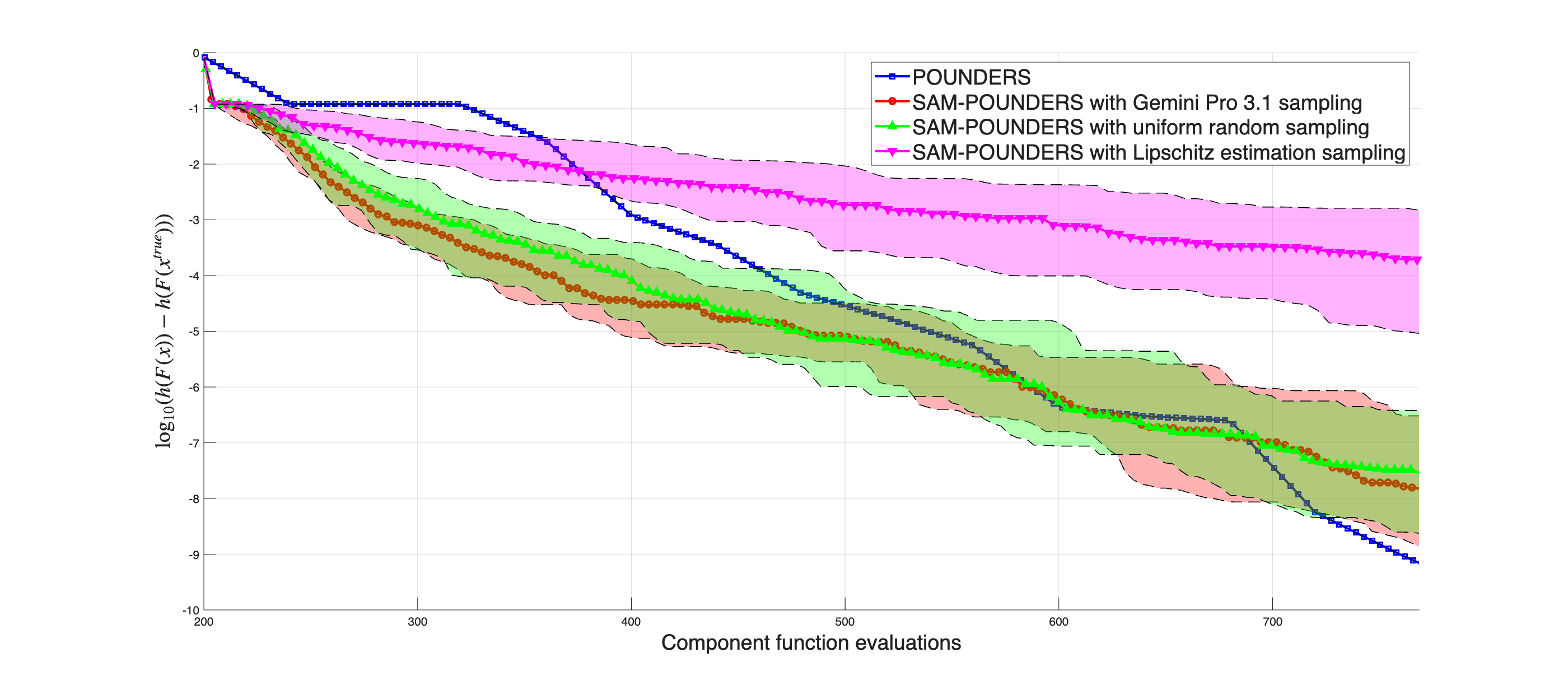}
\end{figure}

While this experiment itself is uninteresting from a scientific perspective, it does suggest an interesting future.
A 1D damped harmonic oscillator is such a textbook example in physics that it is inevitable that it was included in a training set for any commerical LLM; in this sense, the specific use of Gemini Pro 3.1 was unimportant here. 
However, imagining a future with foundation models trained specifically on curated scientific datasets, the framework presented here represents a significant opportunity to insert knowledge and intuition at the level of a domain expert into an agentic workflow.
Agents can search and retrieve relevant literature, and other agents can draw on prior theoretical results, experimental observations, and existing simulation studies to establish a current state of knowledge. 
Synthesizing this evidence, a foundation model (with appropriate tool calling capabilities to handle proposed score calculations) can generate contextually relevant expert advice within \texttt{SAM-POUNDERS}. This has the potential to yield significant acceleration in time to solution by exploiting known scientific results potentially ignored by the problem modeler.    

\section{Conclusion}\label{sec:conclusion} 
In this manuscript, we revisited the framework of stochastic average model (SAM) methods introduced in \cite{SAMP2022}, relevant for problems that involve a finite sum (and certain other compositions beyond sums) of $p$ computationally expensive summands. 
We identified that, for the purposes of analysis, the method in \cite{SAMP2022} made some difficult assumptions on the accuracy of ameliorated estimators that had practical implications.
In particular, the analysis assumed the availability of an optimal, but unknowable, sampling distribution that was only coarsely approximated in practice. 
Misestimation of this distribution -- which could happen occasionally in practice -- resulted in failure, as in the deliberately engineered example illustrated in \Cref{fig:failure}. 
To protect against such failures, and informed by the assumptions on the ameliorated estimator, the SAM method proposed in \cite{SAMP2022} employed a dynamic batch size $b_k$ that satisfied $b_k\to p$ as $k\to\infty$, guaranteeing eventual convergence. 
However, generating such conservative adaptive sequences was dissatisfying in practice; given constrained computational resources and known parallelization capabilities, one would prefer to work with a fixed batch size if possible. 

In the same example of failure in \Cref{fig:failure}, employing instead a naive sampling distribution (uniform sampling) would have prevented non-convergence. 
This informed our consideration of adopting a contextual bandit method (Exp4), which quantifiably curbs regret with respect to a baseline sampling distribution (in particular, uniform sampling), see \Cref{thm:sublinear_regret}. 

This fairly straightforward fix made employing a fixed batch size in SAM methods feasible in practice, as demonstrated in our experiments with \texttt{SAM-POUNDERS} that mixed the Lipschitz estimation-informed sampling distribution proposed in \cite{SAMP2022} with uniform sampling according to Exp4, see \Cref{sec:lipschitz_exp}. 
This already resolved the main practical concern that we brought into the work discussed in this manuscript. 
However, the generality of an expert in the contextual bandits framework opened interesting pathways for future research. 

When alternative, and significantly less expensive, surrogate models (the $\{\sigma_j: j=1,\dots,p\}$ of \Cref{sec:surrogate}) of $F(\xb)$ are available, they can be employed to construct experts with very little computational overhead. 
These surrogates are relevant to settings in computational science where prior data may already exist in similar, but not exactly identical, regimes to that described by $F(\xb)$.
As a particularly structured example, such a setting encompasses the use of low-fidelity models derived from $F(\xb)$, but employing looser convergence tolerances or coarser mesh grids. 
These surrogates are also relevant to settings where specific emulators were already trained using data from similar, but not necessarily identical, regimes. 
Of significant note is that the approach described in this manuscript does not require any specific modeling of, for example, a prior/posterior distribution describing plausible $F(\xb)$, or any direct mapping between fidelity levels. 
Those types of modeling generally involve significant human engineering. 
The SAM method of this manuscript depends entirely on bandit feedback, dynamically rewarding useful experts, while penalizing non-useful experts; all experts are valid until sufficient evidence is generated to deter their use. 
As a somewhat extreme example of this, we considered in \Cref{sec:gemini} the idea of letting AI models directly suggest sampling distributions for SAM methods, given sufficient context about the computational workflow. 
While we intentionally chose a toy example which was almost guaranteed to result in relatively strong performance by the AI expert, this does suggest interesting work that can be done within scientific workflows, particularly when foundation models are fine-tuned with specific domain science knowledge.
In an extension of the work described in this manuscript, an agentic workflow can be designed to fetch relevant context about, and related data to, science questions, monitor the progress of a SAM method, propose sampling distributions based on a synthesis of this data, and execute tool calls to compute sampling distributions. 

\section*{Funding} 
This material is based upon work supported by the U.S. Department of Energy, Office of Science, Office of Advanced Scientific Computing Research Contract
Nos.\ DE-AC02-06CH11357, and by the Scientific Discovery through Advanced Computing (SciDAC) Program through the FASTMath Institute.

\section*{Data availability statement}
Code to reproduce all experiments and figures in this paper are available at the static repository \url{https://github.com/mmenickelly/bandit_sampounders}.
However, as mentioned in the paper, because Gemini Pro 3.1 is a cloud-based LLM provider, not even the author will be able to reproduce exactly the results of \Cref{sec:gemini}, however it is likely that regenerated results will be characteristically similar. 


\begin{appendix}
    \section{Proof of \Cref{thm:sublinear_regret}}
    \begin{proof}
    For notational convenience, define $q_{k,n} := w_{k,n}/W_k$
    and define $y_{k,n} := \mathbf{e}^{k,n\top}\hat{\mathbf{d}}^k$
    for all $k\in\{1,\dots,K\}$ and for all $n\in\{1,\dots,N\}$. 
    Observe that for all $k\in\{1,2,\dots,K-1\}$, 
    \begin{align*}
    \displaystyle\frac{W_{k+1}}{W_k} 
    = \displaystyle\sum_{n=1}^N \frac{w_{k+1,n}}{W_k}
    = \displaystyle\sum_{n=1}^N q_{k,n} \exp\left(\frac{\gamma y_{k,n}}{p} \right)
    & \leq \displaystyle\sum_{n=1}^N q_{k,n} \left[1 + \frac{\gamma y_{k,n}}{p} + (e-2)\left(\frac{\gamma y_{k,n}}{p}\right)^2\right]\\
    & = 1 + \left(\displaystyle\frac{\gamma}{p}\right)\displaystyle\sum_{n=1}^N q_{k,n}y_{k,n} + (e-2)\left(\displaystyle\frac{\gamma}{p}\right)^2 \displaystyle\sum_{n=1}^N q_{k,n} y_{k,n}^2. 
    \end{align*}

The inequality holds because one can derive, by observing that the functions $\exp(z)$ and $1 + z + (e-2)z^2$ intersect at $z=0$ and $z=1$ and comparing these functions' first and second derivatives, that for all $z\in[0,1]$, $\exp(z)\leq 1+z+(e-2)z^2$; here, we denote $e=\exp(1)$.  

Then,
\begin{align*}
    \displaystyle\ln\left(\frac{W_K}{W_1}\right) 
= \ln\left(\prod_{k=1}^{K-1}\frac{W_{k+1}}{W_k}\right)
& \leq \ln\left(\prod_{k=1}^{K-1} \left[1 + \left(\displaystyle\frac{\gamma}{p}\right)\displaystyle\sum_{n=1}^N q_{k,n}y_{k,n} + (e-2)\left(\displaystyle\frac{\gamma}{p}\right)^2\displaystyle\sum_{n=1}^N q_{k,n} y_{k,n}^2\right]\right)\\
& \leq \ln \left(\prod_{k=1}^{K-1} \exp\left(\left(\displaystyle\frac{\gamma}{p}\right)\displaystyle\sum_{n=1}^N q_{k,n}y_{k,n} + (e-2)\left(\displaystyle\frac{\gamma}{p}\right)^2\displaystyle\sum_{n=1}^N q_{k,n} y_{k,n}^2\right) \right)\\
& = \displaystyle\sum_{k=1}^{K-1} \left[\left(\displaystyle\frac{\gamma}{p}\right)\displaystyle\sum_{n=1}^N q_{k,n}y_{k,n} + (e-2)\left(\displaystyle\frac{\gamma}{p}\right)^2\displaystyle\sum_{n=1}^N q_{k,n} y_{k,n}^2\right]\\
& = \left[\displaystyle\frac{\gamma}{p}\sum_{k=1}^{K-1} \sum_{n=1}^N q_{k,n}y_{k,n}\right] + \left[(e-2)\left(\frac{\gamma}{p}\right)^2 \sum_{k=1}^{K-1}\sum_{n=1}^N q_{k,n} y_{k,n}^2\right],
\end{align*}
where the second inequality is because $1+z\leq \exp(z)$ for all $z$, an immediate consequence of the Taylor expansion of $\exp(z)$.

By the definition of $W_K$ and $W_1$, it holds for the $n$th expert that 
$$\ln\left(\frac{W_K}{W_1}\right) \geq \ln(w_{K,n}) - \ln(N)  =
\ln\left(\prod_{k=1}^{K-1}\exp\left(\frac{\gamma y_{k,n}}{p}\right) \right)- \ln(N)
= \displaystyle\frac{\gamma}{p}\sum_{k=1}^{K-1} y_{k,n}- \ln(N). 
$$
Combining the above two displays, 
$$\displaystyle\frac{\gamma}{p}\sum_{k=1}^{K-1}y_{k,n} - \ln(N) \leq \left[\displaystyle\frac{\gamma}{p}\sum_{k=1}^{K-1} \sum_{n=1}^N q_{k,n}y_{k,n}\right] + \left[(e-2)\left(\frac{\gamma}{p}\right)^2 \sum_{k=1}^{K-1}\sum_{n=1}^N q_{k,n} y_{k,n}^2\right], $$
or equivalently,
\begin{equation}\label{eq:int1}
\left[\displaystyle\sum_{k=1}^{K-1}y_{k,n}\right] - \left[\frac{(e-2)\gamma}{p} \sum_{k=1}^{K-1}\sum_{n=1}^N q_{k,n} y_{k,n}^2\right] - \frac{p\ln(N)}{\gamma}
\leq \sum_{k=1}^{K-1} \sum_{n=1}^N q_{k,n}y_{k,n}.
\end{equation}
Using the definition of $y_{k,n}$, it holds for any $k\in\{1,\dots,K\}$ that
\begin{align}\label{eq:int2}
    \displaystyle\sum_{n=1}^N q_{k,n} y_{k,n} &
    =\displaystyle\sum_{n=1}^N q_{k,n} \left(\sum_{j\in I_k}\mathbf{e}^{k,n}_j \hat{\mathbf{d}}^k_j\right)
    = \displaystyle\sum_{j\in I_k}\hat{\mathbf{d}}^k_j\left[\sum_{n=1}^N q_{k,n} \mathbf{e}^{k,n}_j\right] \\\nonumber
    & = \displaystyle\sum_{j\in I_k} \hat{\mathbf{d}}^k_j\left(\frac{\pi^k_j - \frac{\gamma b}{p}}{1-\gamma}\right)
    = \displaystyle\sum_{j\in I_k} \frac{\mathbf{d}_j}{\pi^k_j}\left(\frac{\pi^k_j - \frac{\gamma b}{p}}{1-\gamma}\right)
    \leq \frac{1}{1-\gamma}\displaystyle\sum_{j\in I_k} \mathbf{d}_j,
\end{align}
where the equality at the line break is due to the update rule for $\pi_j^k$ in \Cref{alg:exp4}. 

We now aim to upper bound $\displaystyle\sum_{n=1}^N q_{k,n} y_{k,n}^2$ in terms of $\displaystyle\sum_{j\in I_k} \hat{\mathbf{d}}^k_j$. 
Let $k\in\{1,\dots,K\}$ and let $n\in\{1,\dots,N\}$ be arbitrary. 
First, using an application of Cauchy-Schwarz inequality,
\begin{align*}
y_{k,n}^2 
= \left(\displaystyle\sum_{j\in I_k} \mathbf{e}^{k,n}_j \hat{\mathbf{d}}^k_j\right)^2
= \left(\displaystyle\sum_{j\in I_k} \sqrt{\mathbf{e}^{k,n}_j} \cdot \sqrt{\mathbf{e}^{k,n}_j}\hat{\mathbf{d}}^k_j \right)^2
& \leq \left(\sum_{j\in I_k}\mathbf{e}^{k,n}_j\right)
\left(\sum_{j\in I_k}\mathbf{e}^{k,n}_j(\hat{\mathbf{d}}^k_j)^2 \right)\\
& = b \sum_{j\in I_k}\mathbf{e}^{k,n}_j(\hat{\mathbf{d}}^k_j)^2,
\end{align*}
by the definition of $\mathbf{e}^{k,n}_j$. 
Now taking the weighted sum over all $N$ experts,
\begin{align}\label{eq:int3}
\displaystyle\sum_{n=1}^N q_{k,n} y_{k,n}^2 
& \leq b \sum_{n=1}^N q_{k,n} \sum_{j\in I_k}\mathbf{e}^{k,n}_j(\hat{\mathbf{d}}^k_j)^2
= b\displaystyle\sum_{j\in I_k}(\hat{\mathbf{d}}^k_j)^2\displaystyle\sum_{n=1}^N q_{k,n}\mathbf{e}^{k,n}_j
=  b\displaystyle\sum_{j\in I_k} (\hat{\mathbf{d}}^k_j)^2 \left(\displaystyle\frac{\pi^k_j - \frac{\gamma b}{p}}{1-\gamma} \right)\\\nonumber
& \leq \frac{b}{1-\gamma}\displaystyle\sum_{j\in I_k} \pi^k_j (\hat{\mathbf{d}}^k_j)^2
= \frac{b}{1-\gamma}\displaystyle\sum_{j\in I_k} \pi^k_j \frac{\mathbf{d}^k_j}{\pi^k_j}\hat{\mathbf{d}}^k_j
\leq \frac{bD}{1-\gamma}\displaystyle\sum_{j\in I_k} \hat{\mathbf{d}}^k_j.
\end{align}

Combining \cref{eq:int1}, \cref{eq:int2}, and \cref{eq:int3}, we have that for each expert $n\in\{1,\dots,N\}$, 
\begin{equation}\label{eq:ineq1}
\sum_{k=1}^K \sum_{j\in I_k} \mathbf{d}_j
\geq \left[(1-\gamma)\displaystyle\sum_{k=1}^{K-1}y_{k,n}\right] - \left[\frac{(e-2)bD\gamma}{p} \sum_{k=1}^{K-1}\sum_{j\in I_k}\hat{\mathbf{d}}^k_j\right] - \frac{p(1-\gamma)\ln(N)}{\gamma}.
\end{equation}
We now consider the expectation, with respect to the filtration generated by draws of $I_k$, of both sides of \cref{eq:ineq1}. 
Notice that for the $n$th expert, for each $k\in\{1,\dots,K\}$, and conditioned on the draws $\{I_j: 1 \leq j<k-1\}$, 
$$\mathbb{E}\left[y_{k,n}\right] = \mathbb{E}\left[\sum_{j=1}^p\mathbf{e}^{k,n}_j \hat{\mathbf{d}}^{k}_j \right] = \sum_{j=1}^p \mathbf{e}^{k,n}_j \mathbf{d}^{k}_j.$$
Without loss of generality, suppose the uniform expert is indexed as $n=1$. 
Then, considering the second term in the left hand side of \cref{eq:ineq1}, 
$$\frac{1}{p}\mathbb{E}\left[\displaystyle\sum_{k=1}^K \sum_{j=1}^p \hat{\mathbf{d}}^k_j \right]
=\displaystyle\sum_{k=1}^K\sum_{j=1}^p \frac{1}{p}\mathbf{d}^k_j = C^{K,1} \leq \displaystyle\max_{n=1,\dots, N} C^{K,n}. 
$$
Thus, from \cref{eq:ineq1}, we have, for all $n$ experts, 
$$
\mathbb{E}\left[\sum_{k=1}^K \sum_{j\in I_k} \mathbf{d}_j\right]
\geq (1-\gamma) C^{K,n} - 
(e-2)bD\gamma \displaystyle\max_{n=1,\dots,N} C^{K,n} 
- \frac{p(1-\gamma)\ln(N)}{\gamma}
$$
and so
$$
\mathbb{E}\left[\sum_{k=1}^K \sum_{j\in I_k} \mathbf{d}_j\right]
\geq (1 - ((e-2)bD + 1)\gamma) \displaystyle\max_{n=1,\dots,N} C^{K,n} 
- \frac{p(1-\gamma)\ln(N)}{\gamma},
$$
from which the theorem follows. 
\end{proof}

\section{Reasoning Text Generated by Gemini Pro 3.1}
\begin{figure}[h!]\caption{\label{fig:reasoning} Reasoning text provided by Gemini Pro 3.1 on a random iteration of \texttt{SAM-POUNDERS}. Some \LaTeX \hspace{1pc} formatting performed, otherwise text is verbatim. }
    \begin{framed}
    \tiny
To determine the probability with which to update each quadratic model's center, we need to prioritize models that are most "out-of-date" or "inaccurate" at the current incumbent parameter estimate. The accuracy of a local quadratic model (derived from a second-order Taylor expansion) degrades as the distance from its center point increases and as the underlying function's nonlinearity (higher-order derivatives) becomes more pronounced.
     
     The displacement of a damped harmonic oscillator is given by the formula:
     $$y(t, A, d, w, p) = A \exp(-dt) \cos(wt + p)$$
     
     Where:
     \begin{itemize}
    \item $A$ corresponds to `amplitude`
     \item $d$ corresponds to `damping coefficient`
     \item $w$ corresponds to `angular frequency`
     \item $p$ corresponds to `phase offset`
     \item $t$ corresponds to `time`
     \end{itemize}
     
     Let the incumbent parameters be $P_{inc} = (A_{inc}, d_{inc}, w_{inc}, p_{inc}) = (0.4, 0.6, 0.6, 0.6)$.
     Let the current center parameters for the $i$-th model be $P_{current,i} = (A_i, d_i, w_i, p_i)$ and its time point be $t_i$.
     
     We calculate a "score" for each model, reflecting the urgency of its update. This score is based on a weighted sum of the absolute differences between the model's current center parameters and the incumbent parameters. The weights are the magnitudes of the partial derivatives (sensitivities) of the displacement function with respect to each parameter, evaluated at the incumbent parameters and the specific time point $t_i$.
     
     The magnitudes of the partial derivatives (sensitivities) are:
     \begin{enumerate}
     \item  Sensitivity to Amplitude ($S_A$):
         $S_A(t, P_{inc}) = |\partial y/\partial A|_{(t, P_{inc})} = \exp(-d_{inc}t)  |\cos(w_{inc}t + p_{inc})|$
     
     \item  Sensitivity to Damping Coefficient ($S_d$):
         $S_d(t, P_{inc}) = |\partial y/\partial d|_{(t, P_{inc})} = A_{inc} t  \exp(-d_{inc}t)|\cos(w_{inc}t + p_{inc})|$
     
     \item  Sensitivity to Angular Frequency $(S_w)$:
         $S_w(t, P_{inc}) = |\partial y/\partial w|_{(t, P_{inc})} = A_{inc}  t  \exp(-d_{inc}t) |\sin(w_{inc}t + p_{inc})|$
     
     \item  Sensitivity to Phase Offset ($S_p$):
         $S_p(t, P_{inc}) = |\partial y/\partial p|_{(t, P_{inc})} = A_{inc} \exp(-d_{inc}t) |\sin(w_{inc}t + p_{inc})|$
         \end{enumerate}
     
     The score for each time point $i$ is then calculated as:
     $Score_i = S_A(t_i, P_{inc})  |A_i - A_{inc}| + S_d(t_i, P_{inc}) |d_i - d_{inc}| + S_w(t_i, P_{inc}) |w_i - w_{inc}| + S_p(t_i, P_{inc}) |p_i - p_{inc}|$
     
     \textbf{Rationale for this scoring method:}

     \begin{itemize}
     \item   \textbf{Distance to Incumbent:} The terms $|Param_i - Param_{inc}|$ directly account for how far the current model center is from the incumbent parameters. A larger distance implies the current quadratic approximation is less likely to be accurate at the incumbent point.
     \item   \textbf{Nonlinear Relationship \& Quadratic Model Accuracy:}
     \begin{itemize}
         \item  The sensitivities (weights) capture the local steepness of the displacement function with respect to each parameter. If the displacement is highly sensitive to a parameter (large $S_k$), then even a small difference $|Param_i - Param_{inc}|$ can lead to a significant error in the model's prediction, thus requiring a higher update priority.
         \item  Crucially, the sensitivities $S_d$ and $S_w$ are proportional to $t$. This means that at larger time values ($t_i$), the displacement is more sensitive to changes in `damping coefficient` and `angular frequency`. This also correlates with the magnitudes of the second-order derivatives (Hessian terms) for $d$ and $w$ which scale with $t^2$. A larger magnitude in the Hessian indicates greater curvature, meaning the quadratic model becomes less accurate more rapidly as one moves away from its center. By incorporating $t$ into the sensitivities for $d$ and $w$, models corresponding to later time points, where these parameters have a more pronounced non-linear effect, are naturally given higher priority for updating if their $d$ or $w$ parameters are not aligned with the incumbent.
     \end{itemize}
     \end{itemize} 
     Finally, these scores are normalized by their sum across all 40 time points to produce a probability array where all entries are non-negative and sum to 1.0.
     \end{framed}
     \end{figure}
\end{appendix}

\framebox{\parbox{\columnwidth}{The submitted manuscript has been created by UChicago Argonne, LLC, Operator of Argonne National Laboratory (`Argonne'). Argonne, a U.S. Department of Energy Office of Science laboratory, is operated under Contract No. DE-AC02-06CH11357. The U.S. Government retains for itself, and others acting on its behalf, a paid-up nonexclusive, irrevocable worldwide license in said article to reproduce, prepare derivative works, distribute copies to the public, and perform publicly and display publicly, by or on behalf of the Government.  The Department of Energy will provide public access to these results of federally sponsored research in accordance with the DOE Public Access Plan. \url{http://energy.gov/downloads/doe-public-access-plan}.}}
\end{document}